\documentclass[reqno]{amsart}
\usepackage{amsmath, amsthm, amsfonts, amssymb}
\usepackage{mathrsfs,graphicx}
\usepackage{mathtools}
\providecommand{\noopsort[1]{}}
\usepackage{bbm}
\usepackage{dsfont}
\usepackage[dvipsnames]{xcolor}
\usepackage{mathbbol}
\usepackage{ifthen}
\numberwithin{equation}{section}
\usepackage{a4}
\usepackage[colorlinks=true,linkcolor=blue, citecolor=OliveGreen]{hyperref}
\usepackage{enumerate}
\usepackage{tikz}

\newtheorem{thm}{Theorem}[section]
\newtheorem{cor}[thm]{Corollary}
\newtheorem{prop}[thm]{Proposition}
\newtheorem{lem}[thm]{Lemma}

\theoremstyle{remark}
\newtheorem{rem}[thm]{Remark}

\newtheorem{example}[thm]{Example}

\theoremstyle{definition}
\newtheorem{defn}[thm]{Definition}

\newcommand{\coloneqq}{\mathrel{\mathop:}=}

\newcommand{\eps}{\varepsilon}

\newcommand{\one}{\mathds{1}}

\newcommand{\R}{\mathds{R}}
\newcommand{\C}{\mathds{C}}

\newcommand{\N}{\mathds{N}}

\newcommand{\cF}{\mathscr{F}}
\newcommand{\cX}{\mathscr{X}}

\newcommand{\cL}{\mathscr{L}}
\newcommand{\cB}{\mathscr{B}}
\newcommand{\cM}{\mathscr{M}}

\newcommand{\md}{\mathbb{d}}
\newcommand{\x}{\mathbb{x}}
\newcommand{\y}{\mathbb{y}}
\newcommand{\z}{\mathbb{z}}

\let\k\undefined
\newcommand{\k}{\mathbb{k}}
\let\S\undefined
\newcommand{\E}{\mathbb{E}}
\newcommand{\T}{\mathbb{T}}
\newcommand{\D}{\mathbb{D}}
\newcommand{\S}{\mathbb{S}}
\newcommand{\Df}{\mathbb{D}_{\mathsf{full}}}
\newcommand{\A}{\mathbb{A}}
\newcommand{\Af}{\mathbb{A}_{\mathsf{full}}}
\newcommand{\af}{A_{\mathsf{full}}}

\let\P\undefined
\newcommand{\P}{\mathbb{P}}

\newcommand{\vt}{\vartheta}
\newcommand{\Th}{\Theta}

\newcommand{\dt}{\partial_{\mathbb{t}}}
\newcommand{\fM}{\mathscr{M}}

\newcommand{\cLbp}{\mathscr{L}_{\mathsf{bp}}}

\DeclareMathOperator{\im}{\mathrm{im}}

\begin{document}
\title[Martingales via evolutionary semigroups]{Martingales and Path-Dependent PDE\lowercase{s} \\ via Evolutionary Semigroups}

\author{Robert Denk}
\email{robert.denk@uni-konstanz.de}
\author{Markus Kunze}
\email{markus.kunze@uni-konstanz.de}
\author{Michael Kupper}
\email{kupper@uni-konstanz.de}
\address{Universit\"at Konstanz, Fachbereich Mathematik und Statistik, 78357 Konstanz, Germany}

\begin{abstract}
In this article, we develop a semigroup-theoretic framework for the analytic characterization of martingales with path-dependent terminal conditions. Our main result establishes that a measurable adapted process of the form
\[
V(t) - \int_0^t \Psi(s)\, ds
\]
is a martingale with respect to an expectation operator $\E$ if and only if a time-shifted version of $V$ is a mild solution of a final value problem involving a path-dependent differential operator that is intrinsically connected to $\E$. We prove existence and uniqueness of strong and mild solutions for such final value problems with measurable terminal conditions using the concept of evolutionary semigroups. To characterize the compensator $\Psi$, we introduce the notion of the $\E$-derivative of $V$, which in special cases coincides with Dupire’s time derivative.  We also compare our findings to path-dependent partial differential equations in terms of Dupire derivatives such as the path-dependent heat equation.
\end{abstract}



\maketitle

\section{Introduction}

\subsection*{Markov processes, martingales and PDEs}

There is a strong and natural connection between analysis and stochastics, which becomes particularly evident in the study of \emph{Markov processes}. To illustrate this, let $(Z_t)_{t\geq 0}$ be a homogeneous Markov process with Polish state space $X$, defined on a probability
space $(\Omega, \Sigma, P)$. The associated \emph{transition semigroup} $(S(t))_{t\geq 0}$, which acts on the space $B_b(X)$ of
bounded measurable functions on $X$, encodes the transition probabilities of the Markov process and is linked to the process $(Z_t)_{t\geq 0}$
via the relation
\begin{equation}
\label{eq.Markov1}
E[f(Z_T)\mid  \sigma (Z_s : s\leq t)] = E[f(Z_T) \mid \sigma (Z_t)] = [S(T-t)f](Z_t)\quad \mbox{a.s.}
\end{equation}
for all $0\leq t\leq T$ and $f\in B_b(X)$. From an analytic perspective, the semigroup $S$ offers a link to partial differential equations (PDEs)
via its (infinitesimal) \emph{generator} $A$, defined by $Au \coloneqq \lim_{h\downarrow 0} h^{-1}(S(h)u-u)$ for those functions $u$ for which
this limit exists. Indeed, the unique mild solution of the Cauchy problem $\partial_t u(t) = Au(t)$ with $u(0) = f$ is given by
$u(t) = S(t)f$.

Next, fixing $f\in C_b(X)$, we define 
\begin{equation}
\label{eq.martinale}
M(t) \coloneqq E[ f(Z_T) \mid  \sigma(Z_s : s\leq t)]
\end{equation}
for $t\in [0,T]$. By the tower property of conditional expectation, $M$ is a \emph{martingale}. It follows from
Equation \eqref{eq.Markov1} that $M(t)= u(t, Z_t)$, where $u(t,x) = [S(T-t)f](x)$. Consequently, $u$ solves the \emph{final
value problem} (FVP)
\begin{equation}
\label{eq.fvp1}
    \left\{
    \begin{aligned}
        \partial_tu(t,x) & = -Au(t,x) \quad \mbox{ for } (t,x) \in [0,T)\times X,\\
        u(T,x) & = f(x) \quad \mbox{ for } x\in X.
    \end{aligned}
    \right.
\end{equation}
The identity $M(t) = u(t,Z_t)$ provides both an analytic representation of the martingale and, conversely, a stochastic representation
of the solution to the FVP. 
If $(B_t)_{t\geq 0}$ is a one-dimensional Brownian motion, in which case $Au = \tfrac{1}{2}\partial_x^2 u$, we can obtain Equation \eqref{eq.fvp1}
also directly via It\^{o}'s formula, which yields that $M(t) = u(t, B_t)$ is a martingale if and only if $u$ solves
the (backward) heat equation
\[
\partial_tu(t,x) = -\frac{1}{2}\partial_x^2 u(t,x) \quad\mbox{ for } (t,x) \in [0,T)\times \R
\]
with terminal condition $u(T,x) = f(x)$ for $x\in \R$.

For more information concerning the  connection between Markov processes and martingales on the one hand and semigroups, their generators and PDEs on the other hand, we refer the reader to \cite{cw05, ek, Feller52, Feller54, Jacob05, sv79}.

\subsection*{Path-dependent functions}

The representation $M(t) = u(t,Z_t)$ above applies only to martingales with terminal condition $M(T) = f(Z_T)$. However, in applications
such as finance, one is also interested in pricing \emph{path-dependent derivatives}, where the terminal condition is given by $F(Z)$ for some payoff functional $F$ that depends on the entire path $(Z_t)_{t \in [0,T]}$. To address such cases, we consider the path space $\cX$, consisting of all continuous functions $\x \colon \R \to X$, where $(\x(t))_{t \le 0}$ describes the past and $(\x(t))_{t \ge 0}$ the future evolution. Given $F \in C_b(\cX)$ that depends on the path segment $(\x(t))_{t \in [0,T]}$, we consider the associated martingale
\[
M(t) := E\left[F(Z) \mid \sigma(Z_s : s \le t)\right],
\]
which, by definition, depends on the observed path $(Z_s)_{s \in [0,t]}$. Hence, an almost sure representation of the form $M(t) = u(t, (Z_s)_{s \in [0,t]})$ requires that $u \colon [0,T] \times \cX \to \R$ is defined on the path space $\cX$, rather than the state space $X$. One fruitful approach to link the martingale to a \emph{path-dependent partial differential equation} (PPDE) is through the \emph{functional Itô formula}; see \cite{Cont-Fournie10,ContFournie2023,Cont-Perkowski19,dupire}.
Assuming sufficient smoothness of the terminal condition \cite[Section 4]{Peng-Song-Wang23}, in the case of one-dimensional Brownian motion, this leads to the path-dependent heat equation
\begin{equation}
\label{eq.ppde}
\dt u(t,\x) = -\tfrac{1}{2} \partial^2_{\x} u(t,\x) \quad \text{for } (t,\x) \in [0,T) \times C([0,T];\R),
\end{equation}
with terminal condition $u(T,\x) = F(\x)$, where $\dt$ and $\partial_{\x}$ denote the \emph{Dupire horizontal} and \emph{vertical} derivatives, respectively; see \cite{dupire,Peng-Song-Wang23}. Sobolev solutions for path-dependent equations were introduced in~\cite{PengSong2015}, viscosity solutions have been studied in~\cite{BLT2023,cosso_russo_22,EKTZ2014,zhou_23}, while~\cite{BT2023} provides some regularity results. PPDE involving
Dupire derivatives also appear in the study of Markovian integral equations \cite{kalinin20}.
\medskip

From an analytic perspective, it is natural to ask whether solutions of PPDEs can be represented in terms of semigroups. Such representations not only provide structural insight but also allow the application of powerful tools from semigroup theory; see \cite{EngelNagel2000, Pazy1983}. 
We point out that in Equation \eqref{eq.ppde}, on the left-hand side not the classical time derivative $\partial_t$, but the Dupire time derivative
$\dt$ appears. This is due to the fact that, as a function on the path space, $u$ is typically not differentiable in $t$. This is already the case
in the classical example above, where $M(t) = u(t, B_t)$. Indeed, as the paths of Brownian motion are almost surely not differentiable,
the map $t\mapsto u(t, B_t)$ is typically not differentiable in time, even if $u$ is smooth.

At first sight, this observation makes the representation via semigroups impossible, as these are typically differentiable. 
Nevertheless, in Section \ref{sec:ver Dupire}, 
we show that a suitably shifted version of the function $u$ representing the martingale
can be characterized in terms of a so-called \emph{evolutionary semigroup}, a concept which was recently introduced in \cite{dkk}.

\subsection*{Evolutionary semigroups}
Broadly speaking, evolutionary semigroups extend the notion of transition semigroups associated with Markov processes to more general, possibly non-Markovian, stochastic processes. This is achieved by enlarging the state space so that the original processes becomes Markovian in an extended sense. In this framework, the state space is taken to be the path space restricted to the past, denoted by $\cX^-$, consisting of all paths $\x \in \cX$ that remain constant after time zero. As shown in \cite[Theorem 4.5]{dkk}, any evolutionary semigroup on $B_b(\cX^-)$ is of the form
\begin{equation}\label{eq:itro semigroup}
\S(t)F = \E \Theta_t F,
\end{equation}
where $\E$ is an \emph{expectation operator} and $\Theta_t$ is the \emph{shift operator}. Here, $\Theta_t$
is defined by $[\Theta_tF](\x) \coloneqq F(\vt_t\x)$ for $\x\in \cX$ and $t\in \R$, where  $[\vt_t\x](s) = \x(t+s)$.
As for $\E$, it turns out that $\E^\x[F] := \E[F](\x)$ is a classical expectation of $F \in B_b(\cX)$, describing future probabilistic behavior conditioned on the past path $\x \in \cX^-$.
Expectation operators naturally arise from deterministic and stochastic evolution equations. 
For instance, in the case of Brownian motion, the expectation operator is given by $\E[F](\x) = E[F(\x \otimes_0 B)]$, where $\x \otimes_0 B$ denotes the concatenation of a path $\x$ with Brownian motion $B$ at time $0$; see Example~\ref{ex.wieneroperator}. 
Further examples include deterministic dynamics such as delay differential equations, classical Markovian evolutions, and  stochastic (delay) differential equations, as discussed in \cite[Subsections~6.1--6.3]{dkk}.
We point out that for deterministic and stochastic delay equations there are also alternative semigroup approaches on path-like spaces 
available in the literature, see \cite{bp05, bddm07, ct20, dvv95, mohammed84}.

\subsection*{Main results: martingales and FVPs}

Given an expectation operator $\E$ associated with an evolutionary semigroup $\S$, we naturally arrive at a canonical notion of a martingale:
A measurable adapted process $V: [0,T]\times \cX \to \R$ is called an \emph{$\E$-martingale}, if
\[
V(s) = \E_s V(t)\quad\mbox{ for all } 0\leq s \leq t\leq T,
\]
where $\E_s = \Theta_s\E\Theta_{-s}$ represents the conditional expectation given the path up to time $s$. The process $V$ is called
adapted (or nonanticipative) if for every $t$ the value $V(t,\x)$ depends only on the path segment $(\x(s))_{s\leq t}$. 
As discussed in Section \ref{sect.martingale}, see in particular Lemma \ref{l.pmartingale}, $\E$-martingales are closely related
to the classical notion of a martingale: If $V$ is an $\E$-martingale, then it is a classical $\E^\x$-martingale for every $\x\in \cX^-$,
and the two notions are equivalent under a full support condition.\smallskip

As we will see, $\E$-martingales are closely related to solutions of final value problems of the form
\[
\partial_t U(t) = - \A U(t) + \Phi(t) \quad \mbox{for } t\in [0,T), \quad U(T) = F,
\]
where $\A$ denotes the generator of the evolutionary semigroup $\S$. As a matter of fact, to be able to allow for measurable (rather than continuous) martingales, we have to consider $F$ and $\Phi$ in $B_b(\cX)$, rather than $C_b(\cX)$. For this, we have to work with the so-called
\emph{full generator} $\Af$, introduced in \cite{ek} for transition semigroups of Markov processes and extended to our setting in Definition \ref{d.transition} below. As $\Af$ is typically multivalued, the above FVP has to be understood as the differential inclusion
\[
\partial U(t) \in - \Af U(t) + \Phi(t) \quad \mbox{ for } t \in [0,T), \quad U(T) = F.
\]
Consequently, we have to extend the classical theory for initial value problems for generators of strongly continuous semigroups
found, e.g.,  in \cite[Section 3.1]{abhn}, \cite[Section II.6]{EngelNagel2000} or \cite[Chapter 4]{Pazy1983}. We extend this theory in two main directions:
from initial to terminal conditions and from strongly continuous semigroups to merely measurable semigroups. While the former is 
rather straightforward, the latter is non-trivial due to the lack of continuity and the fact that the relevant generator is multi-valued.
Proposition \ref{p.existence.af.fvp} establishes existence of mild solutions via the usual variation of constants formula. 
Uniqueness is more subtle and holds only up to a suitable equivalence relation, see Theorem \ref{t.mildsolution}. However, continuous mild solutions are indeed unique. We also give characterizations when a mild solution is actually a strong solution, see Theorems~\ref{t.cb.fvp} and~\ref{t.mildstrong}.\smallskip

In our main result, Theorem \ref{t.compensated}, we prove that the process
\begin{equation}
    \label{eq.martintro}
M(t) \coloneqq V(t) - \int_0^t \Psi(r)\, dr
\end{equation}
is an $\E$-martingale, if and only if the function $U(t) \coloneqq \Th_{-t}V(t)$ is, up to equivalence, a mild solution of the FVP
\begin{equation}
\label{eq.fvpintro}
\partial_t U(t) = - \A U(t) + \Phi(t)\quad \mbox{for } t\in [0,T), \quad U(T) = F,
\end{equation}
where $\Phi(t) = \Th_{-t}\Psi(t)$ and $F= \Th_{-T}M(T)$.

In the particular case where $U\equiv F$ and $\Phi\equiv G$, we see in Corollary \ref{c.continuousmartingaleMarkov} that
the process $\Th_t F - \int_0^t\Th_s G\, ds$ is an $\E$-martingale if and only if $F\in D(\Af)$ and $G \in \Af F$.  
This gives a martingale characterization
of the full generator and thus extends \cite[Proposition 4.1.7]{ek} beyond the Markovian setting. However, even within the classical 
Markovian framework discussed in Section \ref{sect:Markov}, the connection between martingales and inhomogeneous FVP yields additional insights:
By Corollary~\ref{c.continuousmartingaleMarkov}, the process $u(t, \x(t)) - \int_0^t \varphi(r, \x(r))\, dr$ is an $\E$-martingale if and only if 
$u$ is a mild solution of $\partial_t u(t) = -A u(t) + \varphi(t)$, where $A$ is the generator of the transition semigroup $S$ obtained by 
restricting $\S$ to functions of the form $f(\x(0))$. Note that, similar to \eqref{eq.fvp1}, this is an equation on the state space $X$.

The martingale characterizations obtained in this work are particularly relevant to martingale problems (see \cite{ek, sv79}), 
which will be further investigated within the current path-dependent and non-Markovian framework in future research. 
For backward SDEs \cite{PardouxPeng1990} with non-Markovian terminal conditions, the link in \cite{PengWang2016} to PPDEs indicate also a natural connection with the FVP \eqref{eq.fvpintro}, which will be explored further. In a nonlinear context, a one-to-one correspondence between convex expectations and convex semigroups was already established in~\cite{ck25}.

\subsection*{Connection to PPDEs with Dupire derivatives}

We have already pointed out that formulating an PPDE for an adapted process $V$ is delicate, as the map $t\mapsto V(t)$ is generally not
differentiable. In the approach via the functional It\^{o} formula, this problem is solved by switching from the usual time derivative
$\partial_t$ to the Dupire time derivative $\dt$. In our approach, we switch to the time-shifed process
$U(t) = \Th_{-t}V(t)$, which satisfies a FVP and is classically time-differentiable at least in the case of a strong solution,
see Theorem \ref{t.mildstrong}. Since the processes $U$ and $V$ are related through a simple time shift, it is natural to study 
their associated equations in parallel. This is the focus of Section \ref{sec:adapted}.\smallskip

Our central notion is that of the so-called $\E$\emph{-derivative}, associated to a given expectation operator $\E$. For an
adapted process $V$, it is defined by
\[
\partial_t^\mathbb{E}V(t)\coloneqq \lim_{h\downarrow 0} \E_t \Big[\frac{V(t+h)-V(t)}{h}\Big].
\]
If $\E$ is associated to a deterministic evolution, see Example \ref{ex.detmart}, then $\partial_t^\mathbb{E}$ is indeed
a time-derivative: It is the derivative in direction of the `characteristics' prescribed by the evolution.
If $[\E F](\x) = F(\tau(\x))$, where $\tau: \cX \to \cX$ is the stopping map, defined by $[\tau(\x)](t) = \x(t\wedge 0)$, the
$\E$-derivative coincide with Dupire's time derivative $\dt$. In this particularly simple case of deterministic evolution, the characteristics
are constant. But also in the general setting, the $\E$-derivative is an important notion. As we show in Theorem \ref{t:Edifferentiable},
an adapted process $V$ can be compensated to an $\E$-martingale if and only if it is continuously $\E$-differentiable.
In this case, the compensator is exactly the $\E$-derivative.\smallskip

As for the relation of the equations for $U$ and $V$, our main result is Theorem \ref{t.martDupire}, which shows that for
an $\E$-martingale of the form \eqref{eq.martintro}, where $V$ is differentiable in the Dupire sense and $U(t) \coloneqq \Th_{-t}V(t)$
is pointwise continuously differentiable, it holds $\Theta_{-t}V(t)\in D(\A)$ for all $t\in [0,T)$ and
\begin{equation}\label{eq:intro V}
\dt V(t) = \Theta_t(\mathbb{L} - \A)\Theta_{-t} V(t) + \Psi(t) \quad \text{for all } t \in [0,T),
\end{equation}
where $\mathbb{L}$ denotes the generator of the evolutionary semigroup associated with the stopping expectation operator. 
As a corollary (see Corollary \ref{c.martDupire})
we obtain that $\E$-martingales of the form $u(t, \x(t)) - \int_0^t \Psi(s, \x)\,ds$, with 
$u \in C^{1,0}([0,T]\times X)$, automatically satisfy spatial regularity, reflecting the classical smoothing effect of parabolic equations.

Equation~\eqref{eq:intro V} is the analogue for $V$ of the FVP~\eqref{eq.fvpintro} for the process $U$.
In the special case of a martingale (i.e., $\Psi \equiv 0$) and when $\A$ is the generator of Brownian motion, Equation \eqref{eq:intro V}
reduces to the path-dependent heat equation. Thus, in this situation, the operator $\Theta_t(\mathbb{L} - \A)\Theta_{-t}$ equals
$-\tfrac{1}{2} \partial^2_{\x}$, i.e., Dupire's second order vertical derivative. The connection between $\E$-martingales
and PPDEs in terms of Dupire derivatives is further explored in Section \ref{sec:ver Dupire} in an It\^{o}-diffusion setting.

\subsection*{Organization}
The paper is structured as follows. Section~\ref{sec:transition semigroups} reviews transition semigroups with a focus on measurable functions and recalls key results on evolutionary semigroups from~\cite{dkk}. Section~\ref{sec:FVP} investigates inhomogeneous FVPs, 
emphasizing mild and strong solution concepts. 
Section~\ref{sect.martingale} presents the main characterization results for $\E$-martingales and their representation through the FVP, 
which are then specialized to the classical Markovian case in Section~\ref{sect:Markov}. 
Section~\ref{sec:adapted} examines the connection between analytic descriptions of adapted processes, particularly via Dupire derivatives, 
and their transformed versions solving the FVP. Illustrative examples are provided throughout the text rather than in a separate section.

\subsection*{Acknowledgements}
The authors are grateful to David Criens for valuable discussions.

\section{Transition semigroups}\label{sec:transition semigroups}

Throughout, let $(X,d)$ be a Polish space with Borel $\sigma$-algebra $\mathscr{B}(X)$. Let $B_b(X)$ and $C_b(X)$ denote the spaces of bounded measurable functions and bounded continuous functions on $X$, respectively. These spaces, endowed with the supremum norm $\|\cdot\|_\infty$, are Banach spaces. 
Furthermore, the space $\fM(X)$ of bounded signed measures on $X$, with the total variation norm $\|\cdot\|_{\rm TV}$, is also a Banach space. We consider $B_b(X)$ and $\fM(X)$ in duality through the canonical dual pairing
\[
\langle f, \mu \rangle \coloneqq \int_X f\, d\mu.
\]

\subsection{Kernel operators}
A \emph{kernel} on $X$ is a map $k: X \times \cB(X) \to \R$, which satisfies the following conditions: 
\begin{itemize}
\item[(i)] $x \mapsto k(x, A)$ is measurable for all $A \in \cB(X)$.
\item[(ii)] $A \mapsto k(x, A)$ is a signed measure for all $x \in X$.
\item[(iii)] $\sup_{x \in X} |k|(x, X) < \infty$. 
\end{itemize}
Here, $|k|(x, \cdot)$ denotes the total variation of $k(x, \cdot)$. If every measure $k(x, \cdot)$ is a probability measure, then $k$ is called a \emph{Markovian kernel}.
\smallskip

Given a kernel $k$, we can define a bounded linear operator $K$ on $B_b(X)$ by setting
\begin{equation}\label{eq.representation}
[Kf](x) \coloneqq \int_X f(y) k(x, dy) \quad \text{for all } x \in X, f \in B_b(X).
\end{equation}
Likewise, we can define a bounded linear operator $K'$ on $\fM(X)$ by
\begin{equation}\label{eq.representation2}
[K'\mu](A) \coloneqq \int_X k(x, A)\, \mu(dx)\quad \text{for all } A\in \cB(X), \mu \in \fM(X).
\end{equation}
We call operators of this form \emph{kernel operators} and $k$ the \emph{associated kernel}.  The operators $K$ and $K'$ are in duality to each other in the sense that
\[
\langle Kf, \mu\rangle = \langle f, K'\mu\rangle
\]
for all $f\in B_b(X)$ and $\mu \in \fM(X)$. Direct verification shows that the operator norms satisfy  
$ \|K\| = \|K'\| = \sup_{x \in X} |k|(x, X) $;  
see also Example 2.3 and Example 2.4 in \cite{kunze11}.

\smallskip
In Equation \eqref{eq.representation}, it may happen that $Kf$ is continuous for all $f \in C_b(X)$. This holds if and only if the map $x \mapsto k(x, \cdot)$ is continuous with respect to the weak topology $\sigma(\fM(X), C_b(X))$, in which case $K$ defines a bounded linear operator on $C_b(X)$. We  call such an operator a kernel operator on $C_b(X)$. 
We note that any bounded linear operator \( K \) on \( C_b(X) \) given by \eqref{eq.representation} for \( f \in C_b(X) \) can be extended to a bounded linear operator on \( B_b(X) \) by applying \eqref{eq.representation} for \( f \in B_b(X) \).
We will not distinguish between a kernel operator on $C_b(X)$ and 
its extension to $B_b(X)$. \smallskip

We say that a sequence $(f_n)_{n\in \N} \subset B_b(X)$  
converges to $f \in B_b(X)$ \emph{in the bp-sense} if  
$\sup_{n\in \N} \|f_n\|_\infty < \infty$ and $f_n(x) \to f(x)$ for all $x \in X$. Here, the abbreviation bp stands for \emph{bounded and pointwise}
An operator $T$ on $B_b(X)$ is called \emph{bp-continuous} if  
$f_n \to f$ in the bp-sense implies that $Tf_n \to Tf$ in the bp-sense.  
We denote by $\cLbp(V)$   
the spaces of bounded bp-continuous operators on $V\in\{B_b(X),C_b(X)\}$.

\begin{lem}\label{l.kernelop}
A bounded linear operator $K$ on $V\in\{B_b(X),C_b(X)\}$ is a kernel operator if and only if $K$ is bp-continuous.
\end{lem}

\begin{proof}
See \cite[Lemma A.1]{bk24} for  $V = C_b(X)$ and \cite[Lemma 5.1]{akk16} for $V = B_b(X)$.
\end{proof}

\subsection{Transition semigroups on \texorpdfstring{$B_b(X)$}{}} %We start with the basic definition.
\begin{defn}\label{d.transition}
A \emph{transition semigroup} is a family $S=(S(t))_{t\geq 0}\subset \cLbp(B_b(X))$ of Markovian kernel operators satisfying:
\begin{enumerate}
    \item[(i)] $S(t+s) = S(t)S(s)$ for all $t, s \geq 0$, and $S(0) = \mathop{\mathrm{id}}$.
    \item[(ii)] For every $f \in B_b(X)$, the map $(t, x) \mapsto [S(t)f](x)$ is jointly measurable.
\end{enumerate}
The \emph{full generator} of $S$ is the multi-valued operator $\af$, defined by $f \in D(\af)$ and $g \in \af f$ if and only if
\begin{equation}
\label{eq.fullgen}
[S(t)f](x) - f(x) = \int_0^t [S(s)g](x)\, ds \qquad \mbox{for all }t>0, x \in X.
\end{equation}
\end{defn}

As a consequence of condition (ii), the integral on the right-hand side of \eqref{eq.fullgen} is well-defined and measurable as a function of $x$. From now on, all integrals will be understood pointwise unless stated otherwise.
In particular, condition (ii) allows us to define additional time integrals involving the semigroup $S$, such as the \emph{Laplace transform} $(R(\lambda))_{\lambda >0}$ of $S$, which is defined by
\begin{equation}
\label{eq.laplacetransform}
[R(\lambda)f](x) \coloneqq \int_0^\infty e^{-\lambda t}[S(t)f](x)\, dt \qquad \mbox{for all }f\in B_b(X), \lambda >0.
\end{equation}

It is easy to see that $R(\lambda)\in \cLbp(B_b(X))$ and $\|R(\lambda)\| = \lambda^{-1}$. Integrating \eqref{eq.laplacetransform}
with respect to a measure $\mu\in \fM(X)$, it  follows that any transition semigroup is an \emph{integrable semigroup}
on the norming dual pair $(B_b(X), \fM(X))$ in the sense of \cite[Definition 5.1]{kunze11}. It follows from
\cite[Proposition 5.2]{kunze11}, that the family $(R(\lambda))_{\lambda>0}$ is a \emph{pseudo-resolvent},
i.e., it satisfies the \emph{resolvent identity}
\begin{equation}
\label{eq.resolventidentity}
R(\lambda_1) - R(\lambda_2) = (\lambda_2 - \lambda_1)R(\lambda_1)R(\lambda_2)\qquad \text{for all } \lambda_1,\lambda_2>0.
\end{equation}
However, $R(\lambda)$ is typically not injective and, hence, is not the resolvent of a single-valued operator in general. It follows from \cite[Proposition A.2.4]{haase} that it is the unique resolvent of a multi-valued operator, namely the operator $\af$.
The resolvent identity \eqref{eq.resolventidentity} guarantees that the kernel and the range of $R(\lambda)$ do not depend on $\lambda > 0$. Using the equivalence of (i) and (iii) in Lemma \ref{l.fullgen}(a) below, one can verify that
\[
D(\af) = \operatorname{im} R(\lambda) \quad \text{and} \quad \af (0) = \ker R(\lambda)
\]
for all $\lambda > 0$. For more details and additional information on multi-valued operators, we refer to \cite[Appendix A]{haase}.

As the following example illustrates, the full generator of a transition semigroup is generally multivalued.

\begin{example}
    \label{ex.heat}
    Let us consider the \emph{Gaussian semigroup} $G_\infty$ on $L^\infty(\R^d)$, defined by
    \[
    [G_\infty(t)f](x) \coloneqq (2\pi t)^{-\frac{d}{2}} \int_{\R^d} e^{-\frac{|x-y|^2}{2t}} f(y)\, dy
    \]
    for $f\in L^\infty(\R^d)$. It was shown in \cite[Example 3.7.8]{abhn} that $G_\infty$ defines a holomorphic semigroup on $L^\infty(\R^d)$
    which is not strongly continuous. Its generator is given by $\frac{1}{2}\Delta_\infty$, where
    \[
    D(\Delta_\infty) \coloneqq  \{ u\in L^\infty(\R^d) : \Delta u \in L^\infty(\R^d) \}\quad \mbox{and} \quad \Delta_\infty u = \Delta u.
    \]
    Here, $\Delta$ refers to the distributional Laplacian. From the definition, it follows that $G_\infty(t)f \in C_b(\R^d)$
    for all $f\in L^\infty(\R^d)$. This fact implies that $D(\Delta_\infty) \subset C_b(\R^d)$. Moreover, it allows us to lift the 
    semigroup $G_\infty$ to a transition semigroup $G$ on $B_b(\R^d)$, in the sense that if $q: B_b(\R^d) \to L^\infty(\R^d)$ maps a bounded measurable function to its equivalence class modulo equality almost everywhere (a.e.), then $q(G(t)f) = G_\infty(t)q(f)$ for all $f\in B_b(\R^d)$. We refer to \cite[Section 5]{akk16} for a thourough discussion of the relation between $G$ and $G_\infty$, as well as for the verification that $G$ indeed satisfies the assumptions above. Using the defining Equation \eqref{eq.fullgen} and a similar equation
    for the generator of $G_\infty$, it follows that the full generator of $G$ is given by $\frac{1}{2}\Delta_\mathsf{full}$,
    where
    \[
    D(\Delta_\mathsf{full}) = D(\Delta_\infty) \quad\mbox{and}\quad f \in \Delta_\mathsf{full} u \,\mbox{ iff } \, \Delta u = q(f). 
    \]
    In particular, it holds $\Delta_{\mathsf{full}}(0) = \{ f\in B_b(\R^d) : f=0\mbox{ a.e.}\}$. 
\end{example}

The full generator satisfies the following properties:
\begin{lem}\label{l.fullgen}
Let $S$ be a transition semigroup with Laplace transform $(R(\lambda))_{\lambda >0}$ and full generator $\af$.
\begin{enumerate}[\upshape (a)]
\item The following are equivalent:
\begin{enumerate}[\upshape (i)]
\item $f \in D(\af)$ and $g \in \af f$.
\item For every $x \in X$, the function $t \mapsto [S(t)f](x)$ belongs to
$W^{1,\infty}_\mathrm{loc}([0,\infty))$ with $\frac{d}{dt}[S(t)f](x) = [S(t)g](x)$ for almost every $t\in[0,\infty)$.
\item For every $\lambda > 0$, it holds $R(\lambda)(\lambda f - g) = f$, i.e., $(\lambda - \af)^{-1} = R(\lambda)$.
\end{enumerate}
\item For every $f \in B_b(X)$ and $t > 0$, \[\int_0^t S(s)f \, ds \in D(\af)\quad\mbox{and}\quad S(t)f - f \in \af  \int_0^t S(s)f \, ds.\]
\item For $f \in D(\af)$, $g \in \af f$ and $s > 0$, \[S(s)f \in D(\af)\quad\mbox{and}\quad S(s)g \in \af S(s)f.\] Furthermore, the map $s \mapsto S(s)f$ is $\|\cdot\|_\infty$-continuous.
\item The full generator $\af$ is bp-closed, i.e., for $(f_n)_{n\in\N}\subset D(\af)$ and $g_n\in \af f_n$ for all $n\in\N$ with $f_n\to f$ and $g_n\to g$ in the bp-sense, it holds $f\in D(\af)$ with $g\in \af f$.
\end{enumerate}
\end{lem}

\begin{proof}
(a). (ii) is merely a reformulation of (i) whereas the equivalence of (i) and (iii) follows from \cite[Proposition 5.7]{kunze11}(a).

(b). This is \cite[Proposition 5.7]{kunze11}(b).

(c). Let $f \in D(\af)$, $g \in \af f$ and $s > 0$. For $t>0$, we have
\[
S(t)S(s)f - S(s)f = S(s)\int_0^t S(r)g\, dr = \int_0^t S(r)S(s)g\, dr.
\]
Indeed, for every $x\in X$, it holds $S(s)'\delta_x\in \cM(X)$ as $S(s)$ is a kernel operator. Thus, Fubini's theorem yields
\begin{align*}
\Big[S(s)\int_0^tS(r)g\, dr\Big](x) &= \Big\langle \int_0^t S(r)g\, dr, S(s)'\delta_x\Big\rangle \\ 
& = 
\int_0^t \langle S(r)g, S(s)'\delta_x\rangle \, dr  = \int_0^t [S(s)S(r)g](x)\, dr,
\end{align*}
for all $x\in X$.
By definition, this shows $S(s)f\in D(\af)$ and $S(s)g\in \af S(s)f$. As for the second statement, observe that for $0\leq s < t$, we have
\[
\|S(t)f-S(s)f\|_\infty = \Big\|\int_s^t S(s)g\, ds \Big\|_\infty \leq \int_s^t \|g\|_\infty \, ds = (t-s)\|g\|_\infty. 
\]

(d). Using that $S$ consists of Markovian kernel operators, this follows directly from the defining Equation~\eqref{eq.fullgen} and the dominated convergence theorem.
\end{proof}

If $\mathcal{A}$ is a multi-valued operator on $B_b(X)$, we can define its $\sigma$-adjoint $\mathcal{A}'$ as a multi-valued operator on $\mathcal{M}(X)$ as follows. We say that $\mu \in D(\mathcal{A}')$ and $\nu \in \mathcal{A}'\mu$ if and only if
\[
 \langle g, \mu \rangle = \langle f, \nu \rangle \quad \text{for all } f \in D(\mathcal{A}), g \in \mathcal{A} f.
\]
Basic algebraic manipulations show that $R(\lambda, \mathcal{A}) = (\lambda - \mathcal{A})^{-1} \in \cLbp(B_b(X))$ implies $R(\lambda, \mathcal{A}') = (\lambda - \mathcal{A}')^{-1} = R(\lambda, \mathcal{A})'$; see \cite[Corollary A.4.3]{haase}.
We can use this fact for transition semigroups as follows. If $S = (S(t))_{t \geq 0}$ is a transition semigroup, we may also consider the \emph{adjoint semigroup} $S' = (S(t)')_{t \geq 0}$ on $\cM(X)$. This is an integrable semigroup on the norming dual pair $(\cM(X), B_b(X))$, whose Laplace transform is given by the $\sigma$-adjoint $R(\lambda)'$ of the Laplace transform of $S$. Thus, by the above, the full generator of $S'$ is the $\sigma$-adjoint $\af'$ of the full generator $\af$ of $S$. 
Writing $\af f$ for any element of $\af f$ and $\af'\mu$ for any element of $\af'\mu$, we find the usual relation
\begin{equation}
\label{eq.dualitygen}
\langle \af f, \mu\rangle = \langle f, \af'\mu\rangle\quad\mbox{for all } f\in D(\af), \mu \in D(\af').
\end{equation}

Since the above results are consequences of general results concerning integrable semigroups on norming dual pairs, it follows %, mutatis mutandis,
that the results of Lemma \ref{l.fullgen} hold similarly for $\af'$.

\subsection{Transition semigroups on \texorpdfstring{$C_b(X)$}{}} We now consider transition semigroups whose restriction to $C_b(X)$ is continuous in $t$ and $x$. Note that such a semigroup in particular leaves the space $C_b(X)$ invariant.

\begin{defn}
A transition semigroup $S=(S(t))_{t\geq 0}$ is called \emph{$C_b$-semigroup} if for every $f\in C_b(X)$ the map
$(t,x) \mapsto [S(t)f](x)$ is continuous on $[0,\infty)\times X$. If $S$ is a $C_b$-semigroup, its $C_b$-generator $A$ is defined as follows:
for $f,g\in C_b(X)$, it holds 
$f\in D(A)$ and $Af=g$ if and only if
\begin{equation}
\label{eq.cbgen}
[S(t)f](x) - f(x) = \int_0^t [S(s)g](x)\, ds \quad\mbox{for all }t>0, x\in X.
\end{equation}
\end{defn}

By definition, the $C_b$-generator is merely the part of the full generator in $C_b$, i.e., we have $f\in D(A)$ and $Af=g$ if and only if  $f,g\in C_b(X)$, $f\in D(\af)$ and $g\in \af f$. In contrast to the full generator, the $C_b$-generator is single-valued. Indeed, if $f \in C_b(X)$ and \eqref{eq.cbgen} is satisfied with $g\in C_b(X)$ and also with $g$ replaced by $\tilde g \in C_b(X)$, then
\[
\int_0^t[S(s)g](x)\, ds = \int_0^t [S(s)\tilde g](x)\, ds \quad\mbox{ for all } t>0, x\in X.
\]
As $S$ is a $C_b$-semigroup, both integrands are jointly continuous in $t$ and $x$. 
Thus, dividing this equation by $t>0$ and letting $t\to 0$ yields $g(x) = \tilde{g}(x)$ for all $x\in X$. 
We collect some further facts concerning the $C_b$-generator.

\begin{lem}
\label{l.cbgen}
Let $S$ be a $C_b$-semigroup with full generator $\af$ and $C_b$-generator $A$. Then, the following statements hold.
\begin{enumerate}
[\upshape (a)]
\item $A$ is the part of $\af$ in $C_b(X)$, i.e., $f\in D(A)$ and $Af=g$ if and only if
$f,g\in C_b(X)$ and $f\in D(\af)$ with $g\in \af f$. 
%In particular, if $f \in C_b(X)\cap D(\af)$, then $\af f\cap C_b(X)$ is either empty or a singleton.
\item For $f,g\in C_b(X)$, the following are equivalent:
\begin{enumerate}
[\upshape (i)]
\item $f\in D(A)$ and $Af=g$.
\item $S(t)f-f = \int_0^t S(s)g\, ds$ for all $t>0$.
\item 
$\sup_{t\in (0,1]}\big\|\frac{S(t)f-f}{t}\big\|_\infty<\infty$ and $\frac{S(h)f-f}{h}\to g$ pointwise as $h\downarrow 0$.
\end{enumerate}
\item For $f\in C_b(X)$, it holds $\int_0^t S(s)f\, ds \in D(A)$ and $A\int_0^t S(s)f\, ds = S(t)f-f$.
\end{enumerate}
\end{lem}

\begin{proof}
(a) was explained above and for (b) see \cite[Theorem A.5]{bk24}. Combining (a) with Lemma \ref{l.fullgen}
yields (c).
\end{proof}

\begin{rem}
    It follows from \cite[Theorem 4.4]{kunze09} that a transition semigroup $S$ is a $C_b$-semigroup, 
    if and only if for $f\in C_b(X)$ the orbit $t\mapsto S(t)f$ is continuous in the strict (or mixed) topology.
    The latter is also often used in the discussion of transition semigroups. We refer to \cite{gnr24} for
    examples of semigroups continuous in the mixed topology, thus for $C_b$-semigroups.
\end{rem}

\subsection{Evolutionary semigroups}

If $(X,d)$ is a Polish space, then the \emph{path space}
$\cX\coloneqq C(\R; X)$, i.e., the set of all continuous functions
$\x: \R\to X$, is also a Polish space when endowed with the metric
\[
\md (\x, \y) \coloneqq \sum_{n=1}^\infty 2^{-n}\big(1\wedge \sup_{t\in [-n, n]} d(\x(t), \y(t))\big).
\]
We may thus also consider transition semigroup on $B_b(\cX)$. Of particular interest is the \emph{shift group}
$(\Th_t)_{t\in \R}$ which is defined as follows. For $t\in \R$, we define the \emph{shift map} $\vt_t: \cX \to \cX$ by
$[\vt_t\x](s) \coloneqq \x(t+s)$. We then define $\Th_t \in \cLbp(B_b(\cX))$ by
\[
[\Th_t F](\x) \coloneqq F(\vt_t\x).
\]
It is not difficult to see that $(\Th_t)_{t\in \R}$ is a transition semigroup in the sense of Definition \ref{d.transition}. In fact,
it is even a \emph{transition group}, as we can allow arbitrary $t\in \R$ instead of merely $t\geq 0$. The full generator of the shift
group is denoted by $\Df$. Actually, the shift group is a $C_b$-group, whence it has a $C_b$-generator denoted by $\D$.
For proofs of these facts and more information about the shift
group and its full generator and $C_b$-generator, we refer to
\cite[Section 3]{dkk}.\smallskip

Here are some additional mappings on $\cX$ of particular interest.
The \emph{evaluation maps} $\pi_t: \cX \to X$ are given by $\pi_t(\x) \coloneqq \x(t)$ for all $t \in \R$. These maps generate the Borel $\sigma$-algebra, i.e., $\mathscr{B}(\cX) = \sigma(\pi_t : t \in \R)$. Furthermore, we can define some additional $\sigma$-algebras as follows. If $I \subset \R$ is an interval, we set $\mathscr{F}(I) \coloneqq \sigma(\pi_t : t \in I)$ and define 
\begin{align*}
B_b(\cX; \mathscr{F}(I)) &\coloneqq \{ F \in B_b(\cX) : F \text{ is } \mathscr{F}(I)\text{-measurable} \}, \\
C_b(\cX; \mathscr{F}(I)) &\coloneqq \{ F \in C_b(\cX) : F \text{ is } \mathscr{F}(I)\text{-measurable} \}.
\end{align*}
The most important cases are when $I = (-\infty, t]$ and $I = [t, \infty)$ for some $t \in \R$. To simplify notation, we write $\mathscr{F}_t \coloneqq \mathscr{F}((-\infty, t])$. 
Furthermore, the \emph{stopping map} $\tau: \cX \to \cX$ is defined by $[\tau(\x)](t) \coloneqq \x(t\wedge 0)$. Combining $\tau$ with the shift $\vt_t$,
we define $\tau_t \coloneqq \vt_{-t} \circ \tau \circ \vt_t$, referred to as the stopping map at time $t$. 
It turns out that a measurable function $F\colon \cX \to \R$ is $\cF_t$-measurable if and only if $F = F \circ \tau_t$; see \cite[Lemma~2.3(c)]{dkk}.\smallskip

We can now recall the concept of an \emph{evolutionary semigroup}, introduced in \cite{dkk}. Broadly speaking, this is a transition
semigroup which for functions that are $\cF_{t_0}$-measurable for some $t_0<0$, behaves initially like the shift semigroup. Before making
this definition rigorous, let us comment on the additional measurability assumption for functions in $B_b(\cX, \cF_0)$. As we remarked above,
a function $F\in B_b(\cX)$ is $\cF_0$-measurable if and only if $F=F\circ\tau$. If we put $\cX^-\coloneqq \tau(\cX)$, then $\cX^-$ is a closed
subspace of $\cX$ and thus also a Polish space with respect to $\md$. Moreover, its Borel $\sigma$-algebra is merely the trace of $\cF_0$
on $\cX^-$, see \cite[Lemma 2.6]{dkk}. We may thus identify $B_b(\cX; \cF_0)$ and $B_b(\cX^-)$ by means of the map $F\mapsto F\circ \tau$.
We note that since $\tau$ is continuous, this establishes also an isomorphism between $C_b(\cX; \cF_0)$ and $C_b(\cX^-)$.

\begin{defn}
An \emph{evolutionary semigroup} is a transition semigroup $\S = (\S(t))_{t\geq 0}$ on $B_b(\cX;\cF_0)$ such that  $\S(t)\Th_{-t}F = F$ 
for all $F\in B_b(\cX; \cF_0)$ and $t\geq 0$. If $\S$ is additionally a $C_b$-semigroup, we say that $\S$ is an \emph{evolutionary $C_b$-semigroup}.
\end{defn}

\begin{defn}
\label{def.expectation}
An \emph{expectation operator} is an operator $\E \in \cLbp (B_b(\cX))$ that satisfies the following properties:  
\begin{enumerate}
[(i)]
\item $\E F \in B_b(\cX; \cF_0)$ for all $F \in B_b(\cX)$.
\item $\E F = F$ for $F \in B_b(\cX; \cF_0)$.
\end{enumerate}
Given an expectation operator $\E$, we define $\E_t$ by setting $\E_t \coloneqq \Th_t \E \Th_{-t}$. We say that
$\E$ is \emph{homogeneous} if $\E \E_t = \E$ for all $t \geq 0$.
\end{defn}

The structure of evolutionary semigroups is described by the following result. For a more detailed discussion, we refer to \cite{dkk}.
\begin{thm}\label{t.evolutionary}
Let $\S$ be an evolutionary semigroup with full generator $\Af$.
\begin{enumerate}
[\upshape (a)]
\item There exists a homogeneous expectation operator $\E$ such that $\S(t) = \E\Th_t$.
Conversely, if $\E$ is a homogeneous expectation operator, then $\T = (\T(t))_{t\geq 0}$, defined by $\T(t)F = \E\Th_tF$ for all 
$F\in  B_b(\cX; \cF_0)$, is an evolutionary semigroup.
\item It holds that $F\in D(\Af)$ with $G\in \Af F$ if and only if there exists $U\in D(\Df)$ with $V\in \Df U$ such that $F=\E U$ and $G=\E V$.
\end{enumerate}
\end{thm}

\begin{proof}
(a) is \cite[Theorem~4.4]{dkk} whereas (b) is proved in \cite[Proposition~4.9]{dkk}.
\end{proof}

We end this section by introducing the evolutionary semigroup $\mathbb{G}$ related to Brownian motion. More examples of evolutionary semigroups might be found in \cite[Section 6]{dkk}.

\begin{example}
    \label{ex.wieneroperator}
    We let $B =(B_t)_{t\geq 0}$ be a $d$-dimensional Brownian motion defined on a probability space $(\Omega, \Sigma, P)$.
    In this example, we have $X= \R^d$ and $\cX= C(\R; \R^d)$. Given $\x\in \cX$ and $\omega\in \Omega$, we define
    \[
    [\x\otimes_0B(\omega)](t) \coloneqq \begin{cases}
        \x(t), & t\leq 0\\
        \x(0) + B_t(\omega), & t>0. 
    \end{cases}
    \]
We define the operator $\mathbb{W} \in \cL(B_b(\cX))$ by setting, for $F \in B_b(\cX)$,
\begin{equation}
    \label{eq.wienerop}
    [\mathbb{W} F](\x) \coloneqq E\bigl[F(\x \otimes_0 B)\bigr],
\end{equation}
where $E$ denotes expectation with respect to the probability measure $P$. Naturally, $\mathbb{W}$ is intimately related to the Wiener measure $W$ 
on $C([0,\infty); \R^d)$. It can be proved that $\mathbb{W}$ is a homogeneous expectation operator, see \cite[Theorem 6.10]{dkk} and the subsequent Example 6.11 therein. We denote the associated evolutionary semigroup
    by $\mathbb{G}$, i.e., $\mathbb{G}(t) = \mathbb{W}\Th_t$. 
\end{example}

\section{Final value problems}\label{sec:FVP}
In this section, we discuss the connection between transition semigroups and time-dependent partial differential equations involving their full generators. 
In the context of strongly continuous semigroups, these connections are well known, see \cite[Section II.6]{EngelNagel2000} and \cite[Section 3.1]{abhn}.
We adopt the relevant notions to our more general setting. However, motivated by our intended applications to martingales in the next section, 
we consider final value problems with a time horizon $T > 0$, rather than initial value problems. 

\subsection{The measurable case} 
Throughout, we fix a final time horizon $T > 0$ and let $S$ be a transition semigroup on $B_b(X)$ with full generator $\af$. 
As before, $(X,d)$ is a Polish space. We denote by $\mathfrak{L}^1(0,T; B_b(X))$ the space of all measurable functions $\varphi : [0,T) \times X \to \R$ such that $\int_0^T \|\varphi(t)\|_\infty \, dt < \infty$.  
The main difference between $\mathfrak{L}^1(0,T;B_b(X))$ and the Bochner space $L^1(0,T;B_b(X))$ is the measurability assumption on its members.
For example, if $T=1$, $X= \R$ and $\varphi(t,x) = \one_{(0,t)}(x)$, then $\varphi \in \mathfrak{L}^1(0,1; B_b(\R))$, but
$\varphi \not\in L^1(0,T; B_b(\R))$, since the map $t\mapsto \varphi (t)$ is not strongly measurable as $\varphi$ is not almost separably valued.

We also define
\[
\mathfrak{L}^1(0,T; C_b(X)) \coloneqq \left\{ \varphi \in \mathfrak{L}^1(0,T; B_b(X)) : \varphi(t) \in C_b(X) \text{ for all } t \in (0,T) \right\}.
\]

We now introduce the central object of this section. Given $f\in B_b(X)$ and $\varphi \in \mathfrak{L}^1(0,T; B_b(X))$, we consider the \emph{final value problem} (FVP)
\begin{equation}
\label{eq.af.fvp}
\left\{
\begin{aligned}
\partial_t u(t) & \in -\af u(t) + \varphi(t)\\
u(T) & = f.
\end{aligned}
\right. 
\end{equation}

\begin{defn}
A \emph{mild solution} of the FVP \eqref{eq.af.fvp} is a function $u\in B_b([0,T]\times X)$ with $u(T) =f$ such that
for all $t\in [0,T)$, it holds $\int_t^Tu(s)\, ds \in D(\af)$ with 
\begin{equation}
\label{eq.mildsol}
u(t)-f + \int_t^T\varphi(s)\, ds \in \af \int_t^Tu(s)\, ds.
\end{equation}
\end{defn}

We first prove existence of a mild solution to the FVP \eqref{eq.af.fvp}. To that end, we use the following auxiliary result.

\begin{lem}\label{l.measurable}
Let $\varphi \in \mathfrak{L}^1(0,T; B_b(X))$.
\begin{enumerate}[\upshape (a)]
    \item The map $(t,x) \mapsto [S(t)\varphi(t)](x)$ is measurable.
    \item The map $\Lambda: (t,x) \mapsto \int_t^T [S(s-t)\varphi(s)](x)$ belongs to $B_b([0,T]\times X)$.
    \item If $\varphi \in \mathfrak{L}^1(0,T; C_b(X))$, then the map $\Lambda$ from {\upshape (b)} belongs to $C_b([0,T]\times X)$.
\end{enumerate}
\end{lem}

\begin{proof}
(a). Let   
\[
\mathcal{H} := \left\{ \varphi \in B_b([0,T] \times X) : (t,x) \mapsto [S(t)\varphi(t)](x)\mbox{ belongs to } B_b([0,T] \times X) \right\}.
\]  
Then $1_{[t_1, t_2] \times A} \in \mathcal{H}$ for all $0 \leq t_1 \leq t_2 \leq T$ and $A \in \cB(X)$. Moreover, 
$\mathcal{H}$ is closed under addition, scalar multiplication, and pointwise convergence of bounded increasing sequences, as all operators $S(t)$
are bp-continous. By the monotone class theorem, $\mathcal{H} = B_b([0,T] \times X)$. This proves (a) for uniformly bounded $\varphi$. By approximation, we can extend this result to $\varphi \in \mathfrak{L}^1(0,T; B_b(X))$.\smallskip

(c). For $0\leq s < t\leq T$ it is $\|\Lambda (t) - \Lambda (s)\|_\infty \leq \int_s^t\|\varphi(r)\|_\infty \, dr \to 0$ as $t-s\to 0$. Thus, to prove
continuity of $\Lambda$, we only need to prove continuity of $x\mapsto \Lambda (t, x)$ for every fixed $t$. However, as $x \mapsto \varphi(s, x)$
is continuous for every $s$ by assumption, this follows easily by dominated convergence.\smallskip

(b). It follows from (c), that for $\varphi \in C_b([0,T]\times X)$ the function $\Lambda_\varphi$ that maps 
$(t,x)$ to $\int_t^T [S(s-t)\varphi(s)](x)\, ds$ belongs to $C_b([0,T]\times X)$. Thus, the set
\[
\mathcal{H}\coloneqq \big\{\varphi \in B_b([0,T] \times X) : \Lambda_\varphi \in B_b([0,T] \times X) \big\}
\]
contains $C_b([0,T]\times X)$. Similar as in (a), the monotone class theorem yields $\mathcal{H}=B_b([0,T] \times X)$. Once again, we can extend
this result to arbitrary $\varphi \in \mathfrak{L}^1(0,T; B_b(X))$ by approximation.
\end{proof}

The variation of constants formula provides a mild solution:
\begin{prop}\label{p.existence.af.fvp} 
Given $f\in B_b(X)$ and $\varphi\in \mathfrak{L}^1(0,T; B_b(X))$, define
\[
u(t) \coloneqq S(T-t)f - \int_t^TS(r-t)\varphi(r)\, dr
\]
for $t\in [0,T]$. Then, $u$ is a mild solution of the FVP \eqref{eq.af.fvp}.
\end{prop}

\begin{proof}
By Lemma \ref{l.measurable}, the mapping $u$ is well-defined and belongs to $B_b([0,T]\times X)$. Using Fubini's theorem,
\begin{align*}
\int_t^Tu(s)\, ds & = \int_t^TS(T-s)f\, ds - \int_t^T\int_s^T S(r-s)\varphi(r)\, dr\, ds\\
& = \int_t^TS(T-s)f\, ds - \int_t^T\int_t^r S(r-s)\varphi(r)\, ds\, dr.\\
\end{align*}
By Lemma \ref{l.fullgen}(b), $\int_t^TS(T-s)f\, ds \in D(\af)$ and 
\[
S(T-t)f-f\in \af \int_t^TS(T-s)f\, ds
\] for all $t\in [0,T)$.
Similarly, $\int_t^r S(r-s)\varphi(r)\, ds\in D(\af)$ with 
\[
S(r-t)\varphi(r) - \varphi(r) \in \af\int_t^rS(r-s)\varphi(r)\, ds
\] 
for all $t< r \leq T$.
Integrating with respect to $r$, we obtain $
\int_t^T\int_t^rS(r-s)\varphi(r)\, ds\, dr \in D(\af) $ with
\[
\int_t^T S(r-t)\varphi(r)\, dr - \int_t^T\varphi(r)\, dr
\in \af \int_t^T\int_t^rS(r-s)\varphi(r)\, ds\, dr .
\]
Indeed, this follows from the characterization of the full generator in Equation \eqref{eq.fullgen} and Fubini's theorem.

Since $\af$ is linear, it follows that $\int_t^T u(s)\, ds\in D(\af)$ for $t\in [0,T)$ with 
\[
u(t)-f + \int_t^T\varphi(s)\, ds \in \af \int_t^Tu(s)\, ds.
\]
This proves that $u$ is a mild solution of \eqref{eq.af.fvp}.
\end{proof}

A natural question is whether mild solutions are unique and in what sense.
\begin{example}
Consider the FVP \eqref{eq.af.fvp} with $\af = \frac{1}{2}\Delta_{\mathsf{full}}$, the full generator of the Gaussian semigroup $G$
from Example \ref{ex.heat} in dimension $d=1$. We also set
$\varphi= 0$ and $f=0$. Then, $u(t,x) \coloneqq  \one_{\{T-t\}}(x)$ for $t\in [0,T]$ and $x\in \R$ defines
a mild solution of \eqref{eq.af.fvp} different from zero (which is also a mild solution of \eqref{eq.af.fvp}). We note that
for every $x\in X$, it is $u(t,x) = 0$ for almost every $t\in [0,T]$. Moreover, it is
$u(t) \in \Delta_{\mathsf{full}}(0)$ for all $t\in [0,T]$.
\end{example}

This example motivates the following definition.

\begin{defn}
%Let $\af$ be the full generator of a transition semigroup and $T>0$. 
We say that two functions $u_1, u_2 \in B_b([0,T]\times X)$
are \emph{nearly equal} with respect to $\af$ if the following properties hold:
\begin{enumerate}
[\upshape (i)]
\item For every $x\in X$, it holds $u_1(t,x) = u_2(t,x)$ for almost every $t\in [0,T]$.
\item For every $t\in [0,T]$, it holds $u_1(t)-u_2(t)\in \af (0)$.
\end{enumerate}
If it is clear from the context which full generator we are considering, we simply say that $u_1$ is nearly equal to $u_2$.
\end{defn}

We next show that two mild solutions of \eqref{eq.af.fvp} are nearly equal. In the proof, we will use the following
lemma.
\begin{lem}\label{l.defect}
If $f\in D(\af)\cap \af (0)$, then $f=0$.
\end{lem}

\begin{proof}
We denote the Laplace transform of $S$ by $(R(\lambda))_{\lambda>0}$ and fix $\lambda_0>0$.
As $f\in D(\af)$, we may write $f= R(\lambda_0)g$ for some $g\in B_b(X)$. By the resolvent identity \eqref{eq.resolventidentity},
for every $\lambda >0$,
\[
f= R(\lambda_0)g= R(\lambda)g + (\lambda-\lambda_0)R(\lambda)R(\lambda_0)g = R(\lambda)g,
\]
as $f=R(\lambda_0)g \in \ker R(\lambda) = \af (0)$. As $S$ is Markovian,  $\|R(\lambda)g\|_\infty=\lambda^{-1}\|g\|_\infty$ which tends to $0$ for $\lambda \to \infty$. This implies $f=0$ as claimed.
\end{proof}

\begin{prop}\label{p.uniqueness.af.fvp}
A function $u \in B_b([0,T] \times X)$ is a mild solution of \eqref{eq.af.fvp} with $f = 0$ and $\varphi = 0$ if and only if $u$ is nearly equal to $0$.
\end{prop}

\begin{proof}
If $u$ is nearly equal to zero, then $\int_t^Tu(s)\, ds =0$ by property (i). In particular, $\int_t^Tu(s)\, ds \in D(\af)$. Moreover,
by property (ii), it follows that \[u(t) \in \af (0) = \af \int_t^Tu(s)\, ds.\]
Hence, $u$ is a mild solution.
\medskip

To establish the converse, let $u$ be a mild solution of \eqref{eq.af.fvp} with $f = 0$ and $\varphi = 0$, and define $U(t) \coloneqq \int_t^T u(s) \, ds$. We proceed in several steps.\smallskip

\emph{Step 1}: We prove that if  $\rho: [0,T]\to \fM(X)$ is continuously differentiable 
with respect to the total variation norm, then $t\mapsto \langle U(t), \rho(t)\rangle$ is weakly differentiable with
\[\frac{d}{dt} \langle U(t), \rho(t)\rangle = \langle U(t), \rho'(t)\rangle-\langle u(t), \rho(t)\rangle.\]

Indeed, suppose first that $\rho(t) = \sum_{k=1}^n\psi_k(t)\mu_k$, where $\psi_1, \ldots, \psi_n\in C^\infty([0,T])$ and $\mu_1\ldots, \mu_n
\in \fM(X)$. Then, for $\chi\in C_c^\infty((0,T))$, we obtain
\begin{align*}
\int_0^T \langle U(s), \rho(s)\rangle\chi'(s)\, ds & = \sum_{k=1}^n \int_0^T \langle U(s), \mu_k\rangle \psi_k(s)\chi'(s)\, ds\\
& = -\sum_{k=1}^n \int_0^T\chi(s) \frac{d}{ds} \big[\psi_k(s)\langle U(s), \mu_k\rangle\big] \, ds\\
& = -\sum_{k=1}^n \int_0^T\chi(s)\big[\psi'_k(s)\langle U(s), \mu_k \rangle -\psi(s)\langle u(s), \mu_k\rangle\big]\,ds\\
& = - \int_0^T\chi(s)\big[\langle U(s), \rho'(s)\rangle - \langle u(s), \rho(s)\rangle\big]\,ds.
\end{align*}
As for the general case, we note that there is a sequence $(\rho_n)_{n\in\N}$ of such particular functions, such that $\rho_n \to \rho$ in  the Bochner space $L^1((0,T); \cM(X))$   and $\rho_n' \to \rho'$ in $L^1((0,T); \cM (X))$  (cf.~\cite[Theorem~46.2]{Treves67}), and the above equality extends by continuity.\smallskip

\emph{Step 2}: If $\mu \in D((\af')^2)$ with $\nu \in \af'(\mu) \cap D(\af')$, then the map $t \mapsto S(t)'\mu$ is differentiable with respect to the total variation norm with \[\frac{d}{dt}S(t)'\mu =  S(t)'\nu.\]

To see this, note that by the adjoint version of Lemma \ref{l.fullgen}(c), %the map $t \mapsto S(t)'\mu$ is continuous in the total variation norm. 
as $\nu \in D(\af')$, the map $t \mapsto S(t)'\nu$ is continuous in the total variation norm. This shows that  
\[
t \mapsto S(t)'\mu = \mu + \int_0^t S(s)'\nu \, ds
\]
is differentiable with respect to the total variation norm with claimed derivative.
\smallskip

\emph{Step 3}: We prove $R(\lambda)U(t) = 0$ for all $t\in [0,T)$ and $\lambda>0$.

To that end, fix $\mu\in D((\af')^2)$ with $\nu \in \af'(\mu)\cap D(\af')$ and $t\in [0,T)$. As a result of Step 1 and Step 2, for $s\in (t,T)$, we obtain
\begin{align*}
\frac{d}{ds}\langle S(s-t)U(s), \mu\rangle & = \frac{d}{ds}\langle U(s), S(s-t)'\mu\rangle\\
& = -\langle u(t), S(s-t)'\mu\rangle + \langle U(t), S(s-t)'\nu\rangle\\
& = -\langle u(t), S(s-t)'\mu\rangle  + \langle u(t), S(s-t)'\mu\rangle = 0,
\end{align*}
where $\langle U(t), S(s-t)'\nu\rangle=\langle u(t), S(s-t)'\mu\rangle$ follows from Equation \eqref{eq.dualitygen}, observing that
 $S(s-t)'\nu \in \af'S(s-t)'\mu$ and $u(t)\in \af U(t)$, as $u$ is a mild solution. 
This shows that
\[\langle U(t), \mu\rangle = \langle S(T-t)U(T), \mu\rangle =0\quad\mbox{for all }\mu\in D((\af')^2).\]

 Now let
$\mu\in D(\af')$ and  $\nu \in \af'\mu$. Define $\mu_n\coloneqq nR(n)'\mu \in D((\af')^2)$. Then, $\mu_n \to \mu$ with respect to the total variation norm, as
\[
\mu_n = R(n)'\nu + \mu \to \mu 
\]
since $R(n)'\nu \to 0$ for $n\to \infty$. We obtain that $\langle U(t), \mu\rangle = 0$ for all $\mu \in D(\af')$. Since
$D(\af') = \im R(\lambda)'$, it follows that $\langle R(\lambda)U(t), \nu\rangle = \langle U(t), R(\lambda)'\nu\rangle = 0$
for all $\nu \in \fM(X)$, proving that $R(\lambda) U(t) = 0$, as claimed.\smallskip

\emph{Step 4:} 
We finish the proof.  
By Step 3, it holds that $U(t) \in \af (0) = \ker R(\lambda)$ for all $t \in [0,T)$. However, since $u$ is a mild solution, $U(t) \in D(\af)$ for $t \in [0,T)$, and Lemma \ref{l.defect} yields $U(t) = 0$ for all $t \in [0,T)$.  
As the integral defining $U(t)$ is pointwise, the (scalar) Lebesgue differentiation theorem implies that for every $x \in X$, $u(t,x) = 0$ for almost every $t \in (0,T)$. Moreover, since $u$ is a mild solution, it follows that $u(t) \in \af U(t) = \af (0)$ for all $t \in [0,T)$. This shows that $u$ is nearly equal to 0.
\end{proof}

Combining Proposition \ref{p.existence.af.fvp} with Proposition \ref{p.uniqueness.af.fvp} yields the following result.

\begin{thm}
\label{t.mildsolution}
Given $f\in B_b(X)$ and $\varphi\in \mathfrak{L}^1(0,T; B_b(X))$, 
a function $u\in B_b([0,T]\times X)$ is a mild solution of \eqref{eq.af.fvp} if and only if $u$ is nearly equal to
\[
(t,x) \mapsto [S(T-t)f](x) -\int_t^T [S(r-t)\varphi(r)](x)\, dr.
\]
\end{thm}

\subsection{The continuous case} We now consider the special case of $C_b$-semigroups. Throughout, we fix $T > 0$ and let $S$ be a $C_b$-semigroup with generator $A$. In this case, if $f \in C_b(X)$ and $\varphi \in \mathfrak{L}^1(0,T; C_b(X))$, we consider the FVP
\begin{equation}
\label{eq.cb.fvp}
\left\{
\begin{aligned}
\partial_t u(t) & = -A u(t) + \varphi(t)\\
u(T) & = f.
\end{aligned}
\right.
\end{equation}

\begin{defn}
A \emph{continuous mild solution} of \eqref{eq.cb.fvp}
is a function $u \in C_b([0,T] \times X)$ with $u(T) = f$, which satisfies
$\int_t^T u(s) \, ds \in D(A) $ and
\begin{equation}
\label{eq.mildsol2}
u(t)  =f+ A \int_t^T u(s) \, ds - \int_t^T \varphi(s) \, ds
\end{equation}
for all $t\in[0,T)$.
\end{defn}

For this solution concept, the following characterization holds.
\begin{thm}\label{t.cb.fvp}
Let $f\in C_b(X)$ and $\varphi\in \mathfrak{L}^1(0,T; C_b(X))$.
\begin{enumerate}
[\upshape (a)]
\item If $u$ is a mild soloution of \eqref{eq.af.fvp} which belongs to $C_b([0,T]\times X)$, then $u$ is a continuous mild solution of
\eqref{eq.cb.fvp}.
\item The FVP \eqref{eq.cb.fvp} has a unique continuous mild solution $u$, given by
\begin{equation}
\label{eq.continuous.mild}
u(t) = S(T-t)f - \int_t^T S(r-t)\varphi(r)\, dr.
\end{equation}
\item If $u$ is the continuous mild solution of \eqref{eq.cb.fvp}, then for every $x\in X$, the map
\[ t \mapsto u(t,x) + \Big[A\int_0^t u(s)\, ds\Big](x)
\] 
is weakly differentiable in $t$ with derivative $\varphi(t,x)$.
\item  If $t\mapsto \varphi(t,x)$ is continuous for all $x\in X$, 
then for every $x\in X$, the map from {\upshape (c)} is differentiable in the classical sense.
\end{enumerate}
\end{thm}

\begin{proof}
(a). If $u\in C_b([0,T]\times X)$, then $\int_t^Tu(s)\, ds \in C_b(X)$ for all $t\in [0,T)$. Moreover, as $\varphi \in \mathfrak{L}^1(0,T; C_b(X))$, we also see that $\int_t^T\varphi(s)\, ds \in C_b(X)$ for all $t\in [0,T)$. Together with the assumption $u(t), f\in C_b(X)$,
it follows that the left-hand side in the inclusion \eqref{eq.mildsol} belongs to $C_b(X)$. It follows from Lemma \ref{l.cbgen}(a),
that the validity of \eqref{eq.mildsol}  implies that of \eqref{eq.mildsol2}. This proves (a).\smallskip

(b). Define $u$ by \eqref{eq.continuous.mild}. By Proposition \ref{p.existence.af.fvp}, $u$ is a mild solution of \eqref{eq.af.fvp}.
Since $\varphi \in \mathfrak{L}^1(0,T; C_b(X))$, Lemma \ref{l.measurable}(c) implies that $u \in C_b([0,T]\times X)$. Now
(a) implies that $u$ is a continuous mild solution of \eqref{eq.cb.fvp}.
\smallskip

As for uniqueness, we note that if $u_1$ and $u_2$ are continuous mild solutions of \eqref{eq.cb.fvp}, then they are mild solutions of \eqref{eq.af.fvp}. Hence, Theorem \ref{t.mildsolution} implies that $u_1$ is nearly equal to $u_2$. 
In particular, $u_1(t) - u_2(t) \in \af (0)$ for all $t \in [0,T)$. However, since $u_1(t) - u_2(t) \in C_b(X)$, Lemma \ref{l.cbgen}(a) implies that $u_1(t) - u_2(t) = A(0) = 0$.
\smallskip

(c). This is immediate from the observation that for every $x\in X$,
\[
\Big[u(t)-u(0) + A\int_0^tu(r)\, dr\Big](x) = \int_0^t\varphi(r, x)\, dr 
\]
and the latter belongs to the Sobolev space $W^{1,1}(0,T)$.

(d). This follows from the well-known fact that in dimension one, a weakly differentiable function with continuous derivative is differentiable in the classical sense.
\end{proof}

We next discuss strong solutions of the FVP \eqref{eq.cb.fvp}.  
We denote by  $C_{b,\mathrm{loc}}([0,T) \times X)$ the space of functions $[0,T) \times X \to \R$ that are bounded and continuous on every $[0,T_0] \times X$ with $T_0 \in (0,T)$.

\begin{defn}
A \emph{strong solution} of the FVP \eqref{eq.cb.fvp} is a function $u \in C_b([0,T]\times X)$ such that 
    \begin{enumerate}[(i)]
    \item $u(t) \in D(A)$ for all $t \in [0,T)$, 
    \item the pointwise classical derivative $\partial_t u$ exists on the interval $[0,T)$ 
    and belongs to $C_{b,\mathrm{loc}}([0,T) \times X)$,
    \item $\partial_t u(t)=-A u(t)+\varphi(t)$ for all $t \in [0,T)$ and $u(T)=f$.
    \end{enumerate}
\end{defn}

The next result explains the relationship between continuous mild and strong solutions. 

\begin{thm}
    \label{t.mildstrong}
    Let $\varphi \in C_{b, \mathrm{loc}}([0,T)\times X)\cap \mathfrak{L}^1(0,T; C_b(X))$.
    \begin{enumerate}
    [\upshape (a)]
        \item Every strong solution of \eqref{eq.cb.fvp} is also a continuous mild solution of \eqref{eq.cb.fvp}.
        \item A continuous mild solution $u$ of \eqref{eq.cb.fvp} is a strong solution if and only if the pointwise classical
        derivative $\partial_tu$ exists on $[0,T)$ and belongs to $C_{b,\mathrm{loc}}([0,T)\times X)$.
        \item If $f\in D(A)$ and $\varphi$ has a pointwise classical derivative $\partial_t\varphi$ on $[0,T)$ which belongs
        to $\mathfrak{L}^1(0,T; C_b(X))$, then \eqref{eq.cb.fvp} has a strong solution.
    \end{enumerate}
\end{thm}

\begin{proof}
    (a). Let $u$ be a strong solution of \eqref{eq.cb.fvp} and fix $0\leq t <s < T$. Note that $\partial_ru$ is continuous on 
    $[t, s]$, hence pointwise integrable. It follows that
    \begin{align}
        u(t) & = u(s) - \int_t^s \partial_ru(r)\, dr\notag\\
        & = u(s) - \int_t^s Au(r)\, dr + \int_t^s \varphi (r)\, dr\notag\\
        & = u(s) - A\int_t^su(r)\, dr + \int_t^s \varphi(r)\, dr.\label{eq.tprime}
    \end{align}
    Here the last equality follows from the fact that $\int_t^su(r)\,dr \in D(A)$ and $A\int_t^s u(r)\, dr = \int_t^s Au(r)\, dr$. 
    Indeed, for $\mu \in D(\af')$, it is
    \begin{align*}
    \Big\langle \int_t^s u(r)\, dr , \af' \mu \Big\rangle & = \int_t^s \langle u(r), \af' \mu \rangle\, dr
    = \int_t^s \langle Au(r), \mu \rangle \, dr\\
    & = \Big\langle \int_t^s Au(r)\, dr , \mu \Big\rangle.
    \end{align*}
    As $\mu \in D(\af')$ was arbitrary, the claim follows. We note that, as $s \to T$, we obtain pointwise convergence of $u(s) \to u(T) = f$.
    Moreover, $\int_t^s\varphi(r)\,dr \to \int_t^T\varphi(r)\, dr$ and $\int_t^su(r)\, dr \to \int_t^T u(r)\, dr$ with respect to $\|\cdot\|_\infty$. Thus, upon
    $s\to T$ in \eqref{eq.tprime}, it follows from the closedness of $A$ that $\int_t^Tu(r)\, dr \in D(A)$ and
    \[
    u(t) = f - A\int_t^T u(r)\, dr + \int_t^T \varphi(r)\, dr.
    \]
    This proves that $u$ is a continuous mild solution of \eqref{eq.cb.fvp}.\smallskip

    (b). Let $u$ be a continuous mild solution with $\partial_t u\in  C_{b,\mathrm{loc}}([0,T) \times X)$. 
    It follows from  \eqref{eq.mildsol2} that for $t\in[0,T)$ and small enough $h$
    \[
    \frac{u(t+h)-u(t)}{h}=-A\frac{1}{h}\int_t^{t+h} u(s)\,ds+\frac{1}{h}\int_t^{t+h} \varphi(s)\,ds.
    \]
    Upon $h\to 0$, it is $h^{-1}(u(t+h)-u(t)) \to \partial_tu(t)$, $h^{-1}\int_t^{t+h}\varphi(s)\, ds \to \varphi(t)$ and
    $h^{-1}\int_t^{t+h}u(s)\, ds \to u(t)$, pointwise. By the closedness of $A$, it follows that $u(t)\in D(A)$ and 
    $Au(t)=-\partial_t A(t)+\varphi(t)$ for all $t\in[0,T)$. Differentiability from the left at $t\in (0,T)$ is proved similarly.
    Altogether, we proved that $u$ is a strong solution.\smallskip

    (c). As $f\in D(A)$, $u_1(t)\coloneqq S(T-t)f$ is a strong solution of \eqref{eq.cb.fvp} with $\varphi\equiv 0$ by Lemma \ref{l.cbgen}(b).
    We claim that the function $u_2$, defined by
    \[
    u_2(t) \coloneqq \int_t^T S(s-t)\varphi(s)\, ds
    \]
    is pointwise differentiable on $[0,T)$ with derivative
    \[
    v(t) \coloneqq \int_t^TS(s-t)\varphi'(s)\, ds - S(T-s)f.
    \]
    Once this is proved, it follows that $u\coloneqq u_1+u_2$, the unique continuous mild solution of \eqref{eq.cb.fvp}, is
    pointwise differentiable. Thus, the claim follows from part (b).

    First note that, by substitution, it is
    \[
    v(t) = \int_0^{T-t}S(r)\varphi'(r+t)\, dr - S(T-t)f.
    \]
    Consequently,
    \begin{align*}
        \int_t^T v(s)\, ds & = \int_t^T \int_0^{T-s} S(r)\varphi'(r+s)\, dr\, ds - \int_t^{T}S(T-s)f\, ds\\
        & = \int_0^{T-t}\int_t^{T-r} S(r)\varphi'(r+s)\, ds\, dr - \int_0^{T-t}S(s)f\, ds\\
        & = \int_0^{T-t}S(r)\big[f - \varphi(r+t)\big]\, dr - \int_0^{T-t}S(s)f\, ds\\
        & = - \int_t^T S(s-t)\varphi(s)\, ds = -u_2(t).
    \end{align*}
    Here, the second equality follows from substitution and Fubini's theorem, the third from the fundamental theorem of calculus
    and the last again from subsitution. We have thus proved that $u_2(t) = - \int_t^T v(s)\, ds$. But as $v$ is a continuous function
    by Lemma \ref{l.measurable}(c), it follows that $u_2$ has a pointwise classical derivative.
\end{proof}

\begin{rem}\label{rem:Cloc}
    If $\varphi$ has a pointwise classical derivative $\partial_t\varphi$ on $[0,T)$ that belongs to $\mathfrak{L}^1(0,T; C_b(X))$, i.e.\ in the setting Theorem \ref{t.mildstrong}(c), then $\varphi \in C_{b, \mathrm{loc}}([0,T)\times X)$. Indeed, writing $\varphi$ as an integral over 
    its derivative, this can be proved similar to Lemma \ref{l.measurable}(c).

    On the other hand, a general function $\varphi\in \mathfrak{L}^1(0,T; C_b(X))$ that is jointly continuous in $t$ and $x$ need
    not belong to $C_{b, \mathrm{loc}}([0,T)\times X)$. As an example, consider for $X= \R$ the function $\varphi: [0,T]\times \R \to \R$, defined by
    \[
    \varphi(t,x) =  \begin{cases}
        0, & t=0\\
        \frac{1}{\sqrt{t}}\big(1- |x-t^{-1}|\big)^+, & t>0.
    \end{cases}
    \]
\end{rem}

It is well known that in the realm of strongly continuous semigroup, the homogeneous initial value problem $u' = Au$ with $u(0)=x$
has a classical solution if and only if $x\in D(A)$, see, e.g., \cite[Proposition 3.1.9(h)]{abhn}. We end this section with an example that
shows that this result does not generalize to strong solutions.

\begin{example}
    We consider again the Gaussian semigroup $G$ from Example \ref{ex.heat}. Note that this is even a $C_b$-semigroup (again, see
    \cite[Example 3.7.8]{abhn}), we denote its $C_b$-generator by $\Delta_\mathrm{c}$. As $G$ is holomorphic, it follows
    that $G(t)f\in D(\Delta_{\mathrm{c}})$ for all $f\in C_b(\R^d)$ and $\partial_t G(t)f = \Delta_c G(t)f$ for $t>0$, 
    which is continuous on $(0,\infty)\times \R^d$.
    This implies that for every $f\in C_b(\R^d)$ the final value problem
    \[
    \begin{cases}
        \partial_t u(t) & = - \Delta_\mathrm{c} u(t)\\
        u(T) & = f
    \end{cases}
    \]
    has a strong solution.  
\end{example}

\section{Martingales and their relation to semigroups on path space}
\label{sect.martingale}

Throughout this section, let $\S$ denote an evolutionary semigroup with expectation operator $\E$ and full generator $\Af$. Sometimes, we additionally assume that $\S$ is a $C_b$-semigroup, in which case its generator is denoted by $\A$.
As before, we set $\E_t = \Theta_t \E \Theta_{-t}$. 
We denote the kernel of $\E$ by $\k$ and define $\P^\x \coloneqq \k(\x, \cdot)$ as the corresponding probability measures. We also write $\E^\x$ for the (conditional) expectation with respect to $\P^\x$.
%We denote the kernel associated with $\E$ by $\k$ and write $\P^\x \coloneqq \k(\x, \cdot)$ for the corresponding probability measures. We write $\E^\x$ for (conditional) expectation with respect to $\P^\x$.

We first provide some additional information about the operators
$\E_t$ which shows that they behave like conditions expectations.

\begin{lem}\label{l.expectation}
The following hold true.
\begin{enumerate}
[\upshape (a)]
\item For $F\in B_b(\cX)$ the function $\E_tF$ is $\cF_t$-measurable. If $F$ is $\cF_t$-measurable,
then $\E_t F=F$.
\item For every $G\in B_b(\cX)$ and $F\in B_b(\cX; \cF_t)$, we have $\E_t(FG) = F\E_t G$.
\item It is $\E_s\E_t = \E_s$ for all $0\leq s\leq t$.
\item For every $t\geq 0$, $\x\in\cX$ and $F\in B_b(\cX)$, we have
\[
 [\E_tF](\y) = \big(\mathbb{E}^\x[F\mid\cF_t]\big)(\y) \quad \mbox{for }\P^\x\mbox{-almost all } \y.
\]
\end{enumerate}
\end{lem}

\begin{proof}
    (a). By definition of $\E_t$, it is obvious that $\E_t F$ is $\cF_t$-measurable. If $F$ is $\cF_t$-measurable, then $\tilde F \coloneqq \Th_{-t}F$ is $\cF_0$-measurable, whence $\E\tilde F = \tilde F$ by  Definition~\ref{def.expectation}(i). It follows $\E_tF = F$.\smallskip

    (b). If $F, G$ are $\cF_t$-measurable, then $\tilde F \coloneqq \Th_{-t}F$ and $\tilde G \coloneqq \Th_{-t}G$ are $\cF_0$-measurable.
    By \cite[Proposition 4.1]{dkk}, $\E (\tilde F \tilde G) = \tilde F \E \tilde G$. Applying $\Th_t$, (b) follows.\smallskip

    (c). Homogeneity is equivalent to $\E\Th_t\E= \E\Th_t$ for all $t\geq 0$,
    which  implies
    $\E_s\E_t = \Th_s\E\Th_{t-s}\E\Th_{-t} = \Th_s\E\Th_{t-s}\Th_{-t} = \E_s$ for  all $0\leq s\leq t$.\smallskip

    (d). This is \cite[Proposition 4.7]{dkk}.
\end{proof}

\begin{defn}
Let $I \subset \R$ be an interval. A function 
\[
V \colon I \to B_b(\cX)
\]
(which we frequently identify with the map $V \colon I \times \cX \to \R$) is said to be
\begin{enumerate}[\upshape (a)]
\item \emph{measurable}, if the map $V \colon I \times \cX \to \R$ is measurable,
\item \emph{adapted}, if the function $\x\mapsto V(t,\x)$ is $\cF_t$-measurable for all $t \in I$,
\item an \emph{$\E$-martingale}, if it is measurable, adapted, and satisfies the martingale property
\[
V(s) = \E_s V(t) \quad \text{for all } s, t \in I \text{ with } s \leq t.
\]
An $\E$-martingale $V$ is called \emph{continuous}, if the map $(t, \x) \mapsto V(t, \x)$ from $I \times \cX$ to $\R$ is continuous.
\end{enumerate}
\end{defn}
We remark that for an $\E$-martingale $V$ defined on a compact interval $I$, the martingale property implies that $V \in B_b(I \times \cX)$.

It follows from Lemma \ref{l.expectation}(d), that if $V$ is an $\E$-martingale, then $V$ is a martingale in the classical sense
with respect to every measure $\P^\x$. However, the converse is not true in general:

\begin{example}
Let $\mathbb{W}$ be the expectation operator associated to Brownian motion on $\R$ from Example \ref{ex.wieneroperator} and put $V(t) = \one_{(0,T)}(t)F_t(\one_{\{0\}})$.
Then, $V \in B_b([0,T] \times \cX)$ is adapted and $\mathbb{W} V(t) \equiv 0$, so that $V(t) = 0$ $\P^\x$-almost surely for all $t \in [0,T]$ and $\x \in \cX$.
Hence, $V$ is a martingale with respect to every $\P^\x$. However, for $0 < s < t < T$, it holds that $\mathbb{W}_s V(t) = 0 \neq V(s)$, whence $V$ is not an $\mathbb{W}$-martingale.
\end{example}

The situation improves for continuous martingales under an additional condition on the expectation operator $\E$. 

\begin{defn}\label{def:full-support}
An expectation operator $\E$ has \emph{full support} if for all $F\in C_b(\cX)$, it follows from $\E|F|=0$ that $F=0$.
\end{defn}

Note that $\E$ has full support if and only if for every $F\in C_b(\cX)$ with $F=0$ $\P^\x$-almost surely for every $\x\in \cX$, we already have $F=0$. 
In the context of stochastic differential equations, the support condition plays an important role. The theorem of Stroock--Varadhan \cite[Theorem 3.1]{Stroock-Varadhan72} gives a sufficient condition for full support, see also \cite{ck20} for a generalization to path dependent equations. See
Example \ref{Ex:Support-condition} below for an example of an expectation operator without full support.

\begin{lem}
    \label{l.pmartingale}
    Assume that $\S$ is a $C_b$-semigroup whose expectation operator has full support. Let $V: I \to C_b(\cX)$ be adapted. If $V$ is a $\P^\x$-martingale for all $\x \in \cX$, then $V$ is also an $\E$-martingale.
\end{lem}

\begin{proof}
By assumption and Lemma~\ref{l.expectation}(d), we have
\[
\E_s V(t) = \E^\x\bigl[V(t) \mid \cF_s \bigr] = V(s) \quad \P^\x\text{-almost surely}
\]
for all $s, t \in I$ with $s \leq t$ and $\x \in \cX$. Since $\S$ is a $C_b$-semigroup, it follows that $\E_s V(t) \in C_b(\cX)$; see \cite[Theorem 6.2]{dkk}. Thus, the continuous functions $\E_s V(t)$ and $V(s)$ are equal $\P^\x$-almost surely for all $\x \in \cX$. Hence, $\E\bigl|\E_s V(t) - V(s)\bigr| = 0$ and therefore $\E_s V(t) = V(s)$, showing that $V$ is an $\E$-martingale.
\end{proof}

\begin{example}\label{Ex:4.5}
The expectation operator $\mathbb{W}$ associated to Brownian motion on $\R^d$ satisfies the full support condition
from Definition~\ref{def:full-support}.
To see this, let $0 \leq F \in C_b(\cX)$. We show that if $F \neq 0$, then $\mathbb{W} F > 0$. To that end, suppose that $F(\x) > 0$. By continuity,
we find $\eps, \delta > 0$ such that $F(\y) \geq \eps$ whenever $\md(\x, \y) < \delta$. Setting $S \coloneqq \{\y : \md(\x, \y) < \delta\}$,
it follows that $F \geq \eps \one_S$. Given $\z \in \cX$, we write $\z^+$ for its restriction $\z|_{[0,\infty)} \in \cX^+ = C([0,\infty); \R^d)$.
Denoting the restriction of $\md$ to $\cX^+$ by $\md^+$, we define $S^+ \coloneqq \{ \y^+ : \md^+(\x^+ - \x(0), \y^+) < \delta \}$.
It then follows that
\begin{align*}
    [\mathbb{W} F](\x) &\geq \eps [\mathbb{W} \one_S](\x) = \eps E\bigl[ \x \otimes_0 B \in S \bigr] 
    = \eps E [ B \in S^+] = \eps W(S^+) > 0,
\end{align*}
since $S^+$ is open and the Wiener measure $W$ has full support; see, e.g., \cite[Corollary VIII.2.3]{ry99}.
\end{example}

Example \ref{Ex:Support-condition} below shows that the full support condition is a necessary assumption in Lemma \ref{l.pmartingale}.\smallskip

There is a fundamental link between $\E$-martingales and the evolutionary semigroup $\S$ associated with $\E$. A related result in the context of nonlinear expectations is provided in \cite[Theorem 2.15]{ck25}.

\begin{thm}\label{t.martingale}
Let $I\subset \R$ be an interval and $V \colon I \to B_b(\cX)$ be measurable and adapted.  
Then, $V$ is an $\E$-martingale if and only if, for every $s, t \in I$ with $s \leq t$,  
\begin{equation}\label{eq.martingalesg}  
V(s) = \Th_s \S(t-s) \Th_{-t} V(t).  
\end{equation}  

Setting $U(t) := \Th_{-t} V(t)$ for $t \in I$, so that $U(t)$ is $\cF_0$-measurable for all $t \in I$, Equation \eqref{eq.martingalesg} is equivalent to  
\begin{equation}\label{eq.martingalerephrased}  
U(s) = \S(t-s) U(t)  
\end{equation}  
for all $s, t \in I$ with $s \leq t$.
\end{thm}

\begin{proof}
First assume that $V$ is an $\E$-martingale. Then, for every $s,t\in I$ with $s\leq t$,
\[
V(s) = \E_s V(t) = \Th_s \E \Th_{-s} V(t) = \Th_s\E \Th_{t-s}\Th_{-t}V(t) = \Th_s\S(t-s)\Th_{-t}V(t),
\]
which shows \eqref{eq.martingalesg}. Conversely, assuming \eqref{eq.martingalesg}, for every $s,t\in I$ with $s\leq t$,
\[
\E_sV(t) = \Th_s\E\Th_{-s}V(t) = \Th_s \E\Th_{t-s}\Th_{-t}V(t) = \Th_s \S(t-s)\Th_{-t}V(t) = V(s).
\]
The second part is a straightforward reformulation.
\end{proof}

Combining the previous result with Theorem~\ref{t.mildsolution} and  Theorem~\ref{t.cb.fvp} and Theorem~\ref{t.mildstrong}, 
we immediately obtain the following result:

\begin{cor}
    \label{c.martingalemild}
Let $T>0$, $V\in B_b([0,T]\times \cX)$ be adapted and set $U(t) := \Th_{-t} V(t)$ for all $t \in [0,T]$ as well as $F:=U(T)$.
\begin{enumerate}
    [\upshape (a)]
    \item If $V$ is an $\E$-martingale, then $U$ is a mild solution of the FVP \[\partial_t U(t) \in - \Af U(t)\quad\mbox{with}\quad U(T) = F.\]
    \item Assume additionally that $\S$ is a $C_b$-semigroup and $V\in C_b([0,T]\times \cX)$. Then, $V$ is a continuous $\E$-martingale if and only if $U$ is the unique continuous mild solution of the FVP
    \[
    U(t) = - \A U(t)\quad\mbox{with}\quad U(T) = F.
    \] 
    Moreover, if $F\in D(\A)$, then $U$ is a strong
    solution of this FVP.
\end{enumerate}
\end{cor}

\begin{rem}
We point out that in part (a) of Corollary~\ref{c.martingalemild}, the two statements are not equivalent. This is due to the fact that, in this case, the FVP $\partial_t U(t) \in -\Af U(t)$ with $U(T) = F$ is, in general, not uniquely solvable. Indeed, it follows from Theorem~\ref{t.mildsolution} (with $\varphi = 0$) that $U$ is a mild solution if and only if it is nearly equal to the function $t \mapsto \S(T-t)F$. However, Theorem~\ref{t.martingale} shows that the latter is the only solution of the FVP that gives rise to an $\E$-martingale.
\end{rem}

We next address the question, when an adapted process $V$ can be compensated to become a martingale by 
subtracting an absolutely continuous process $t\mapsto \int_0^t\Psi(s)\, ds$. We note that $\Psi \in \mathfrak{L}^1(0,T; B_b(\cX))$
if and only if $\Phi$, defined by $\Phi(t)\coloneqq \Th_{-t}\Psi(t)$, belongs to $\mathfrak{L}^1(0,T; B_b(\cX))$. In this case, $\Psi$ is adapted
if and only if $\Phi(t)$ is $\cF_0$-measurable for all $t\in (0,T)$.

\begin{thm}
    \label{t.compensated}
    Let $T>0$ and $V\in B_b([0,T]\times\cX)$ and $\Psi \in \mathfrak{L}^1(0,T; B_b(\cX))$ be adapted. 
    Further, for $t\in[0,T]$, define $U(t) \coloneqq  \Th_{-t}V(t)$, $F\coloneqq \Th_{-T}V(T)$ and $\Phi(t) \coloneqq \Th_{-t}\Psi(t)$
    as well as
    \[
    M(t) \coloneqq V(t) - \int_0^t\Psi(s)\, ds.
    \]
    Then, the following are equivalent:
    \begin{enumerate}[\upshape (i)]
        \item $M$ is an $\E$-martingale.
        \item $U(t) = \S(T-t)U(T) - \int_t^T \S(r-t) \Phi(r)\, dr$ for all $t\in[0,T]$.
    \end{enumerate}
    In this case, $U$ is a mild
        solution of the FVP \[\partial_t U(t) \in -\Af U(t) + \Phi(t)\quad\mbox{with}\quad U(T)=F.\]
\end{thm}

\begin{proof}
    (i) $\Rightarrow$ (ii). Assume that $M$ is an $\E$-martingale. Fix $t\in [0,T)$, and define
\[
N(r) \coloneqq \Th_{-t}M(t+r)\quad\mbox{for all }r\in [0,T-t].
\]
Then, $N$ is an $\E$-martingale as for $0\leq s \leq r \leq T-t$,
\begin{align*}
    \E_sN(r) & = \Th_s\E\Th_{-(t+s)}M(t+r) = \Th_{-t}\E_{t+s} M(t+r) = \Th_{-t}M(t+s) = N(s).
\end{align*}
Consequently, it holds $N(0) = \E N(T-t)$, which yields
\begin{align*}
U(t)-\Th_{-t}\int_0^t\Th_s\Phi(s)\, ds & = \E N(T-t) = \E \Big[ \Th_{-t} V(T)  - \Th_{-t}\int_0^{T}\Psi(s)\, ds\Big]\\
& = \E \Big[ \Th_{T-t}U(T) - \int_0^T \Th_{s-t}\Phi(s)\, ds\Big],
\end{align*}
and therefore
\[
U(t)=\S(T-t) U(T)-\int_t^{T} \S(s-t)\Phi(s)\,ds.
\]
\smallskip

(ii) $\Rightarrow$ (i). Assume that $U(t) = \S(T-t)U(T) - \int_t^T \S(t-s)\Phi(s)\, ds$ for all $t\in [0,T]$. 
We define
\[
M(t) = \Th_t U(t) - \int_0^t \Th_s \Phi(s)\, ds.
\]
Then, clearly, $M$ is adapted. Fix $0\leq s \leq t \leq T$. Then 
\begin{align}
& \Th_s\S(t-s)\Th_{-t}M(t)  = \Th_s\S(t-s)U(t) - \E_s \int_0^t \Th_r\Phi(r)\, dr\notag\\
& = \Th_s\S(t-s)U(t) - \int_0^s \Th_r\Phi(r)\, dr - \E_s \int_s^t \Th_r\Phi(r)\, dr,\label{eq.step1}
\end{align}
where we used in the second equality Lemma \ref{l.expectation}(b) and the fact that the integral until time $s$ is $\cF_s$-measurable. 
By assumption, \[U(s) = \S(t-s)U(t)-\int_s^t \S(r-s) \Phi(r)\,dr,\] so that the last term in \eqref{eq.step1} becomes
\[
\E_s \int_s^t \Th_r\varphi(r)\, dr = \Theta_s \int_s^t \S(r-s)\Phi(r)\,dr= \Theta_s \big[\S(t-s)U(t)-U(s)\big].
\]
Plugging this into \eqref{eq.step1} yields
\[
 \Th_s\S(t-s)\Th_{-t}M(t) = \Theta_s U(s)- \int_0^s \Th_r\varphi(r)\, dr =M(s).
\]
By Theorem~\ref{t.martingale}, $M$ is an $\E$-martingale.
\end{proof}

As a corollary, we obtain the following extension of \cite[Proposition 4.1.7]{ek}.

\begin{cor}
    \label{c.martingaleproblem}
    Let $F, G\in B_b(\cX; \cF_0)$ and put
    \[
    M(t) \coloneqq \Th_t F - \int_0^t \Th_s G\, ds
    \]
    for all $t\geq 0$. Then, $M$ is an $\E$-martingale if and only if $F\in D(\Af)$ and $G\in \Af F$.
\end{cor}

\begin{proof}
    It follows from Theorem \ref{t.compensated}, applied to $V(t) = \Th_t F$ and $\Psi(t) = \Th_t G$, that $M$ is an $\E$-martingale 
    if and only if
    \[
    F = \S(T)F - \int_0^T \S(s)G\, ds
    \]
    for all $T>0$. By definition, see Equation \eqref{eq.fullgen}, this is equivalent to $F\in D(\Af)$ and $G\in \Af F$.
\end{proof}

Under additional continuity assumptions, we can further strengthen the above results. We note that in the situation
of Theorem \ref{t.compensated}, it is $\Psi(t) \in C_b(\cX)$ if and only if $\Phi(t)\in C_b(\cX)$ for all $t\in (0,T)$.

\begin{thm}
    \label{t.continuousmartingale}
    In the situation of Theorem \ref{t.compensated}, assume additionally that $\S$ is an evolutionary $C_b$-semigroup and that 
    $F, \Phi(t) \in C_b(\cX)$. Morover, assume that the equivalent conditions of Theorem \ref{t.compensated} are satisfied so that
    $M$ is an $\E$-martingale. 
    \begin{enumerate}
        [\upshape (a)]
        \item $M$ is a continuous $\E$-martingale and $U$ is the unique continuous mild solution of the FVP
            \begin{equation}\label{eq.FVP.A}
            \partial_t U(t) = -\A U(t) + \Phi(t)\quad\mbox{with}\quad U(T) = F.
            \end{equation}  
            Moreover, $\Phi$ is the pointwise weak derivative of 
            \begin{equation}\label{eq.FVP.B}
            t\mapsto U(t) + \A \int_0^t U(s)\, ds.
            \end{equation}
        \item If $\Phi$ is pointwise continuous on $[0,T)$, then $\Phi$ equals the pointwise classical derivative of \eqref{eq.FVP.B}.
        \item If $\Phi\in C_{b,\mathrm{loc}}([0,T)\times \cX)$ and the pointwise classical
        derivative $\partial_t U$ exists on $[0,T)$ and belongs to $C_{b,\mathrm{loc}}([0,T)\times \cX)$, then $U$ is a strong solution of \eqref{eq.FVP.A}.
        \item If $F\in D(\A)$ and $\Phi$ has a pointwise classical derivative $\partial_t\Phi$ on $[0,T)$, which belongs
        to $\mathfrak{L}^1(0,T; C_b(\cX))$, then $U$ is a strong solution of \eqref{eq.FVP.A}.
        \end{enumerate}
        We recall that if $U$ is a strong solution of \eqref{eq.FVP.A}, then $\Phi(t) = \dt U(t) + \A U(t)$ for all $t \in [0,T)$.
\end{thm}

\begin{proof}
    (a).  Using the assumption that $F, \Phi(t)\in C_b(\cX)$, Lemma \ref{l.measurable}(c) implies that
    $U \in C_b([0,T]\times\cX)$. It follows from Theorem \ref{t.cb.fvp}, that $U$
    is a continuous mild solution of \eqref{eq.FVP.A}. This Theorem also
    implies the rest of (a) as well as part (b).

    (c) and (d). These follow from Theorem~\ref{t.mildstrong}(b) and Theorem~\ref{t.mildstrong}(c), respectively.
\end{proof}

We end this section with an illustration of our results. To that end, fix $T>0$ and let $G\in B_b(\cX; \cF_T)$ be given. Then, clearly,
$M(t) \coloneqq \E_t G$ defines an $\E$-martingale.  The connection between this $\E$-martingale and solutions of associated
final value problems is described by Theorem \ref{t.martingale} and Corollary \ref{c.martingalemild}. However, for certain final 
values $G$, we can obtain a characterization via different final value problems by means of Theorem \ref{t.compensated} and
Theorem \ref{t.continuousmartingale}.

Given $f\in B_b(X)$ and $a,b,t \in \R$, we define the functions $F_a^b(f)$ and $F_t(f)$, which are elements of $B_b(\cX)$,
by setting
\[
[F_a^b(f)](\x) \coloneqq \int_a^bf(\x(s))\, ds \qquad\mbox{and}\qquad [F_t(f)](\x) \coloneqq f(\x(t)).
\]
We note that these functions play an important role in the description of the generator of the shift group $(\Th_t)_{t\in \R}$, see
\cite[Section 3]{dkk}.

\begin{example}
    \label{ex.integralmean}
We set $G = F_0^T(f)$ for some $f \in B_b(X)$ and $M(t) \coloneqq \E_tG$. 
Noting that $F_0^T(f) = F_0^t(f) + F_t^T(f)$ and that $F_0^t(f)$  
is $\cF_t$-measurable, we can split the $\E$-martingale as  
\[
M(t) := \E_t F_0^T(f) =  \E_t F_t^T(f)  + F_0^t(f) = V(t) - \int_0^t \Psi(s) \, ds,
\]
where $V(t) := \E_t F_t^T(f)$ and $\Psi(s) = -F_s(f)$. 

Setting, as above, $U(t) \coloneqq \Th_{-t}V(t)$ and $\Phi(s) \coloneqq \Th_{-s}\Psi(s) = -F_0(f)$, it follows from Theorem~\ref{t.compensated} 
that $U$ is a mild solution of the FVP  
\[
\partial_t U(t) \in -\Af U(t) - F_0(f) \quad \mbox{with} \quad U(T) = 0.
\]

Suppose, in addition, that $\S$ is a $C_b$-semigroup with generator $\A$ and $f \in C_b(X)$. Then, we have $\Phi \equiv -F_0(f) \in C_b([0,T] \times \cX) \subset \mathfrak{L}^1(0,T; C_b(\cX))$ and Theorem \ref{t.continuousmartingale} yields that $U$ is the unique mild solution of the FVP  
\[
\partial_t U(t) = -\A U(t) - F_0(f) \quad \mbox{with} \quad U(T) = 0.
\]  
In fact, as $\partial_t\Phi(t) \equiv 0$ and $U(T)=0\in D(\A)$, $U$ is even a strong solution of this equation.
\end{example}  

\section{Examples in the Markovian setting}\label{sect:Markov}

We note that since $\cF(\{0\}) = \sigma(\pi_0)$, a function $F\in B_b(\cX)$ is $\cF(\{0\})$-measurable if and only if
$F=F_0(f)$ for some $f\in B_b(X)$. In this section, we consider evolutionary semigroups that preserve functions of form.

\begin{defn}
    Let $\S$ be an evolutionary $C_b$-semigroup. Then $\S$ is called \emph{Markovian}, if $\S(t) F$ is $\cF(\{0\})$-measurable
    for every $F\in B_b(\cX; \cF(\{0\})$.
\end{defn}

For the rest of this section, $\S$ denotes a Markovian, evolutionary $C_b$-semigroup. We denote its $C_b$-generator by $\A$ and 
its expectation operator by $\E$. As before, we denote by $\P^\x$ the measures derived from the kernel of $\E$. It follows from 
the definition, that a Markovian $C_b$-semigroup $\S$ induces a $C_b$-semigroup $S$ on $X$, by defining $S(t)$ via
\[
\S(t)F_0(f) = F_0(S(t)f)\qquad \mbox{for all } t\geq 0, f\in B_b(X).
\]
Note that the semigroup $S$ is merely the restriction of $\S$ to $B_b(\cX, \cF(\{0\}))$.
The $C_b$-generators $\A$ of $\S$ and $A$ of $S$ are related by  
\begin{equation}\label{eq:DMarkov}
F_0(f) \in D(\A) \quad \text{if and only if} \quad f \in D(A),
\end{equation}  
and we have $\A F_0(f) = F_0(A f)$, see \cite[Theorem 6.13(a)]{dkk}.

\begin{example}
    Let $\mathbb{W}$ be the expectation operator associated to Brownian motion from Example \ref{ex.wieneroperator}, and
    $\mathbb{G}$ be the associated evolutionary semigroup. Then $\mathbb{G}$ is Markovian and the induced semigroup on 
    $C_b(\R^d)$ is the Gaussian semigroup from Example \ref{ex.heat}. We thus have
    \[
    \mathbb{G}(t)F_0(f) = F_0(G(t)f)
    \]
    for all $t\geq 0$ and $f\in B_b(\R^d)$ in this example.
\end{example}

The terminology `Markovian' stems from the fact that, under the measure $\P^\x$, the coordinate process
$(Z_t)_{t\geq 0}$, defined by $Z_t(\y) = \y(t)$, is a Markov process with transition semigroup $S$, 
starting at $\x(0)$, see \cite[Theorem 6.7(b)]{dkk}. For more information on the Markovian setting, 
we refer to \cite[Section 6.2]{dkk}.\smallskip 

When studying $\E$-martingales, or, equivalently, solutions of FVP for $\A$, it will be important for us
that if $\S$ is Markovian, then $\E F$ is $\cF(\{0\})$-measurable whenever $F$ is $\cF([0,\infty))$-measurable. As a matter of fact,
this condition is equivalent to $\S$ being Markovian, see \cite[Theorem 6.7]{dkk}. It is a consequence of this special structure of $\E$
that, at least in some particular cases, we can rewrite certain FVP for $\A$ on the path space $\cX$ as FVP for $A$ 
on the state space $X$. The latter are much easier to handle. We have the following result.

\begin{cor}
    \label{c.continuousmartingaleMarkov}
    Let $\S$ be an evolutionary $C_b$-semigroup that is Markovian and let $T>0$. Moreover, let $\varphi \in \mathfrak{L}^1(0,T; C_b(X))$
    and $u\in C_b([0,T]\times X)$. We set $f= u(T) \in C_b(X)$. Then, the function
        \begin{equation}\label{eq:mart:Markov}
            M(t, \x )= u(t, \x(t)) - \int_0^t\varphi(s, \x(s))\, ds
        \end{equation}
    defines a continuous $\E$-martingale if and only if $u$ is the  unique continuous mild solution of the FVP
        \begin{equation}\label{eq:FVP:A:Markov}
        \partial_t u(t) = -A u(t) + \varphi(t)\quad\mbox{with}\quad u(T) = f.
        \end{equation}
    In that case, $\varphi$ is the pointwise weak derivative of $ t\mapsto u(t) + A \int_0^t u(s)\, ds$. 
    If $u$ is a strong solution of \eqref{eq:FVP:A:Markov}, then $ \varphi(t) = \partial_t u(t) + A u(t)$ for all $t\in[0,T)$.
\end{cor}

\begin{proof}
    Note that $M(t) = V(t) - \int_0^t \Psi(s)\, ds$, where $V(t) = F_t(u(t))$ and $\Psi(t) = F_t(\varphi(t))$. It follows that
    $U(t) \coloneqq \Th_{-t}V(t) = F_0(u(t))$ and $\Phi(t) \coloneqq \Th_{-t}\Psi(t) = F_0(\varphi(t))$. 
    By Theorem \ref{t.compensated}, $M$ is an $\E$-martingale if and only if 
    \[
    U(t) = \S(T-t)U(T) - \int_t^T \S(r-t)\Phi(r)\, dr
    \]
    for all $t\in [0,T]$.
    As $\S$ is Markovian, given the special structure of $U$ and $\Phi$, this reduces to
    \begin{align*}
    F_0(u(t)) & = F_0 (S(T-t)f) - \int_t^T F_0(S(r-t)\varphi(r))\, dr\\
    & = F_0(S(T-t)f) - F_0\Big(\int_t^T S(r-t)\varphi(r)\, dr\Big)
    \end{align*}
    for all $t\in [0,T]$. Noting that $F_0(g) = 0$ as a function on $\cX$ if and only if $g=0$ as a function on $X$, this is equivalent
    to 
    \[
    u(t) = S(T-t)f - \int_t^T S(r-t)\varphi(r)\, dr
    \]
    for all $t\in [0,T]$. By Theorem \ref{t.cb.fvp}(b), $u$ is the unique continuous mild solution of \eqref{eq:FVP:A:Markov}
    and $\varphi$ can be obtained as the pointwise weak derivative of $t\mapsto u(t) + A\int_0^tu(s)\, ds$.
\end{proof}

We now discuss some examples, in which Corollary \ref{c.continuousmartingaleMarkov} can be applied.

\begin{example}\label{ex:finalvalue}
   For $G = F_T(f)$ with $f \in C_b(X)$, we consider $M(t) := \mathbb{E}_t G$ for $t \in [0,T]$. As $\S$ is Markovian, 
   $M(t)$ is $\cF(\{t\})$-measurable, hence of the form $M(t) = F_t(u(t))$ or, written differently, $M(t, \x) = u(t,\x(t))$. By
   Corollary \ref{c.continuousmartingaleMarkov}, $M$ is a continuous $\E$-martingale if and only if $u$ is the unique continous
   mild solution of the FVP
\[
\partial_t u(t) = -A u(t) \quad \text{with} \quad u(T) = f,
\]
hence, by Theorem~\ref{t.cb.fvp}, if $u(t) = S(T - t)f$ for all $t \in [0,T]$. 
\end{example}

\begin{example}
\label{ex.integralmeanmarkovian}
We consider Example~\ref{ex.integralmean} in the Markovian setting.
Let $G = F_0^T(f)$ for some $f \in C_b(X)$. We split $G = F_0^t(f) + F_t^T(f)$
and define $V(t):=\E_t F_t^T(f)$ for all $t\in[0,T]$. As $\S$ is Markovian, $V(t)$ is $\cF(\{t\})$-measurable, hence
of the form $V(t,\x) = u(t, \x(t))$ for some function $u: [0,T]\times X\to \R$. Thus
\[
M(t, \x) = u(t, \x(t)) + \int_0^t f(\x(s))\, ds.
\]
Applying Corollary \ref{c.continuousmartingaleMarkov} with $\varphi \equiv -f$, it follows that $M$ is a continous
$\E$-martingale if and only if 
$u$ is the unique continous mild solution of the FVP  
\[
\partial_t u(t) = -A u(t) - f \quad \mbox{with} \quad u(T) = 0,
\]  
i.e., $u(t) = \int_t^T S(r-t)f\, dr$, see Theorem \ref{t.cb.fvp}.
Similar to Example~\ref{ex.integralmean}, we obtain that $u$ is even a strong solution.
\end{example}    

In our final example, we consider the special case of Brownian motion in space dimension $1$. We thus consider the 
expectation operator $\E$ from Example \ref{ex.wieneroperator} and its associated evolutionary semigroup $\S$.
Recall that the induced semigroup on $C_b(\R)$ is the Gaussian semigroup from Example~\ref{ex.heat}. We write 
\[
p_t(x, y) := \frac{1}{\sqrt{2\pi t}} \exp\left(-\frac{(x - y)^2}{2t}\right).
\]  
Note that the heat kernel $p_t$ is a solution of the heat equation, i.e., $\partial_t p_t(x, y) = \frac{1}{2} \, \partial^2_x p_t(x, y)$
for all $x,y\in \R$ and $t>0$.

\begin{example}
In the setting of Brownian motion, we illustrate our results in the context of the \emph{running maximum}. 
It is well known that  the joint process \((B_t, \sup_{0 \le s \le t} B_s)\) is Markovian \cite[Section III.3]{ry99}, 
and characterizations of local martingales of the form \(H(B_t, \sup_{0 \le s \le t} B_s)\) are provided in~\cite{Obloj06,Obloj06b}.

For $f \in C^1_b(\R)$ and $0 \leq t \leq T$, we define $(V_t)_{t \in [0,T]}$ by  
\[
V(t) := \E_t M_t^T(f), \quad \text{where} \quad M_t^T(f) := f\Big(\max_{s \in [t,T]} \x(s)\Big).
\]  
The Markovian structure yields that for every $t\in[0,T]$, we have
\[
V(t,\x) =u(t,\x(t))\quad\mbox{with}\quad u(t,\cdot)\in C_b(X).
\]
As a consequence of the reflection principle for Brownian motion, see \cite[Proposition III.3.7]{ry99}, it follows that
\[
u(t, x) = 2 \int_{x}^\infty f(y) p_{T-t}(x, y)\,dy,
\]  
for all $x\in\R$ and $t\in[0,T)$. Direct computation shows that
\begin{align*}
\partial_x u(t,x) & = -2f(x)\, p_{T-t}(x,x) + 2 \int_{x}^\infty f(y)\, \partial_x p_{T-t}(x, y)\, dy,\\
\frac{1}{2} \partial^2_x u(t,x)&  = -\frac{f'(x)}{\sqrt{2\pi(T - t)}} + \int_{x}^\infty f(y)\, \partial^2_x p_{T-t}(x, y)\, dy,
\end{align*}
and
\[
\partial_t u(t,x) = -2 \int_{x}^\infty f(y)\, \partial_t p_{T-t}(x, y)\, dy = -\int_{x}^\infty f(y)\, \partial^2_x p_{T-t}(x, y)\, dy.
\]
We conclude that $u(t) = u(t, \cdot)$ is a strong solution of the FVP
\[
\partial_t u(t) = -\frac{1}{2}\partial^2_x u(t) - \frac{f'}{\sqrt{2\pi(T - t)}} \quad \text{with} \quad u(T) = f.
\]
Hence, it follows from Corollary~\ref{c.continuousmartingaleMarkov} that 
\[
M(t, \x)  = u(t, \x(t)) + \int_0^t \frac{f'(\x(s))}{\sqrt{2\pi(T-s)}}\, ds
\]
is a continuous $\E$-martingale. In other words, the stochastic process
\[
E\bigg[f\bigg(\sup_{s \in [t, T]} B_s\bigg) \,\bigg| \,\sigma(B_s : s\leq t) \bigg] + \int_0^t \frac{f'(B_s)}{\sqrt{2\pi(T - s)}}\, ds
\]
is a martingale on $[0, T]$. 
\end{example}

\section{Differentiability of adapted processes}\label{sec:adapted}

In Section \ref{sect.martingale} we have seen that an adapted process $V$ can be compensated to an $\E$-martingale if and only if
the shifted process $U(t) = \Th_{-t}V(t)$ solves a certain final value problem. If $U$ is even a strong solution of
the latter, then $U$ is pointwise differentiable in time. On the other hand, the process $V$ is typically not differentiable in time,
even in the cases where $U$ is.

To see this, let us look at the situation where $V(t, \x) = u(t, \x(t))$, that frequently appears in the Markovian setting. 
In this case, $U(t, \x) = u(t, \x(0))$, and this is differentiable in time if and only if $u$ is pointwise differentiable in time.
On the other hand, for $V$ to be differentiable in time, the limit
\[
\lim_{h \to 0} \frac{V(t + h, \x) - V(t, \x)}{h} = \lim_{h \to 0} \frac{u(t + h, \x(t + h)) - u(t, \x(t))}{h}
\]
needs to exist. However, even if $u$ is smooth, the map $t \mapsto \x(t)$ is merely continuous, so the limit does not exist in general.

\subsection{The \texorpdfstring{$\E$}{}-derivative}

Throughout, we fix an evolutionary $C_b$-semigroup $\S$ with expectation operator $\E$ and $C_b$-generator $\A$.

\begin{defn}
Let $V\colon [0,T] \to C_b(\cX)$ be measurable and adapted. We say that $V$ is \emph{$\mathbb{E}$-differentiable} if, for every $t \in [0,T)$, the limit
\[
\partial_t^{\mathbb{E}} V(t) \coloneqq \lim_{h \downarrow 0} \E_t\left[\frac{V(t+h) - V(t)}{h}\right]
\]
exists in the bp-sense. 
If, in addition, $\partial_t^{\mathbb{E}} V$ belongs to $C_{b,\mathrm{loc}}([0,T) \times \cX)$, we say that $V$ is \emph{continuously $\mathbb{E}$-differentiable}.
\end{defn}

We illustrate this definition in the case of deterministic evolutions, i.e., in the case where the expectation operator
$\E$ is given by 
\begin{equation}
\label{eq.deterministic}
[\E F](\x) = F(\varphi(\x))
\end{equation}
for a map $\varphi : \cX \to \cX$. This situation was studied in \cite[Section 6.1]{dkk}, where it was proved that
Equation \eqref{eq.deterministic} defines a homogeneous expectation operator if and only if $\varphi$ is an \emph{evolution map}, i.e., it holds (i) $\varphi(\tau(\x)) = \varphi(\x)$, (ii) $\tau(\varphi(\x)) = \tau(\x)$ and (iii) $\varphi(\vt_t\varphi(\x)) = \vt_t\varphi(\x)$
for all $t\geq 0$ and $\x\in \cX$, see \cite[Proposition 6.2]{dkk}. Expectation operators of this form arise naturally in the study
of ordinary differential equations or deterministic delay equations, see \cite[Section 6.1]{dkk}.

\begin{example}
    \label{ex.detmart}
    Let $\varphi$ be an evolution map and $\E$ be the associated evolution operator, defined by \eqref{eq.deterministic}. We also set
     $\varphi_t\coloneqq \vt_{-t}\circ\varphi\circ\vt_t$. Then $[\E_tF](\x) = F(\varphi_t(\x))$.

    It follows that a  measurable adapted function $V: [0,T] \to C_b(\cX)$ is an $\E$-martingale if and only if 
    \[
    V(s, \x) = \E_sV(t,\x) = V(t, \varphi_s(\x))\quad \mbox{ for all } 0\leq s \leq t\leq T.
    \]
    In particular, we obtain that $V(0,\x) = V(0, \varphi(\x)) = V(t, \varphi(\x))$ for all $t\in [0,T]$, i.e., $V$ is 
    constant along the orbits prescribed by $\varphi$. This shows that the orbits given by $\varphi$ serve as `characteristics'
    for our evolution.

    We next consider the special situation where $u \in C_b([0,T]\times X)$ and $V(t, \x)= u(t, \x(t))$. It follows that $\E_s V(t) = u(t, \varphi_s(t, \x))$, where we write $\varphi_s(t, \x) \coloneqq [\varphi_s(\x)](t)$. Thus, 
    \begin{align}
        \label{eq.deterministicderivative}
    \E_t\left[\frac{V(t+h) - V(t)}{h}\right] &= \frac{u(t+h, \varphi_t(t+h, \x)) - u(t, \varphi_t(t, \x))}{h}\\
    & = \frac{u(t+h, \varphi(h, \vt_t\x) - u(t, \x(t))}{h}
    \end{align}
    This shows that $V$ is $\E$-differentiable if and only if $u$ is, at every point, differentiable in the direction 
    prescribed by $\varphi$.

    Let us look at a slightly different situation, where we set $\tilde{V}(t,\x) \coloneqq u(t, \varphi(t, \x))$.
    It follows from point (iii) in the definition of an evolution map, that $\varphi_t(\varphi(\x)) = \varphi(\x)$, so that
    Equation \eqref{eq.deterministicderivative} becomes
    \[
    \E_t\left[\frac{\tilde{V}(t+h) - \tilde{V}(t)}{h}\right] = \frac{u(t+h, \varphi(t+h, \x)- u(t, \varphi(t, \x))}{h}.
    \]
    Thus, $\tilde{V}$ is $\E$-differentiable if and only if $u$ is differentiable along the orbits of $\varphi$.
    
    In particular, if $X = \R^d$ and the map $t \mapsto \varphi(t, \x)$ is differentiable for $t \geq 0$ 
    (as is typically the case in applications to delay equations), then for $u \in C^{1,1}([0,T] \times \R^d)$, 
    the function $\tilde{V}$ is $\E$-differentiable and satisfies
    \[
    \partial^\E_t \tilde{V}(t,\x) = (\partial_t u)(t, \varphi(\x, t)) + (\nabla_x u)(t, \varphi(\x, t)) \cdot \varphi'(\x, t).
    \]
\end{example}

We now return to the general situation. Note that it follows directly from the definition that every $\E$-martingale $M\colon [0,T] \to C_b(\cX)$
is $\E$-differentiable with $\partial_t^\E M \equiv 0$. The following theorem gives an interpretation of the $\E$-derivative in the general case.

\begin{thm}\label{t:Edifferentiable}
Let $V\colon [0,T] \to C_b(\cX)$ be measurable and adapted. Then, the following are equivalent:
 \begin{enumerate}[\upshape (i)]
        \item $V$ is continuously $\E$-differentiable with $\partial_t^\E V\in \mathfrak{L}^1(0,T;C_b(\cX))$.
        \item There is an adapted $\Psi\in\mathfrak{L}^1(0,T;C_b(\cX))\cap C_{b,\mathrm{loc}}([0,T)\times \cX)$ such that the process 
        $M(t):=V(t)-\int_0^t \Psi(s)\,ds$ is an $\E$-martingale.
 \end{enumerate}  
 In this case, it holds $\Psi = \partial_t^\E V$ and $U(t) := \Theta_{-t} V(t)$ is the unique continuous mild solution of the FVP
\begin{equation} \label{eq:FVP:A*}
\partial_t U(t) = -\A U(t) + \Theta_{-t} \Psi(t) \quad \text{with} \quad U(T) = \Theta_{-T} V(T).
\end{equation}

\end{thm}

\begin{proof}
(i) $\Rightarrow$ (ii). Fix $0 \le s \le t \le T$ and define $V_s(r):=\E_s V(r)$ for all $s\le r< t$. Since 
\[
\frac{V_s(r+h)-V_s(r)}{h}=\E_s\bigg[\E_r\Big[\frac{V(r+h)-V(r)}{h}\Big]\bigg],
\]
the dominated convergence theorem implies that $r\mapsto V_s(r)$ is right-differentiable with right-derivative $V'_s(r)=\E_s \partial^\E_tV(r)$ and $r\mapsto V'_s(r)$ is pointwise continuous. Hence, we obtain that 
\[
\E_s[V(t)-V(s)]=V_s(t)-V_s(s)=\int_s^t V'_s(r)\,dr=\E_s\Big[\int_s^t \partial^\E_tV(r)\,dr \Big].
\]
This shows that $V(t) - \int_0^t \partial_t^{\mathbb{E}} V(s)\, ds$ is an $\E$-martingale. Hence, $\Psi:=\partial^\E_t V$ is as desired.

(ii) $\Rightarrow$ (i). Assume that $V(t) - \int_0^t \Psi(s)\, ds$ is an $\E$-martingale
for some $\Psi \in \mathfrak{L}^1(0,T; C_b(\cX))\cap C_{b,\mathrm{loc}}([0,T)\times \cX)$.
Then, for all $t \in [0,T)$, we have
\[
\lim_{h \downarrow 0} \mathbb{E}_t\left[ \frac{V(t+h) - V(t)}{h} - \frac{1}{h} \int_t^{t+h} \Psi(r)\, dr \right] = 0.
\]
Thus, $V$ is continuously $\mathbb{E}$-differentiable with $\partial_t^{\mathbb{E}} V = \Psi$.\smallskip

Now assume that the equivalent conditions (i) and (ii) are satisfied. It follows from Theorem \ref{t.compensated}, 
that 
\[
U(t) = \S(T-t)U(T) - \int_t^T\S(r-t)\Th_{-r}\Psi(r)\, dr.
\]
As $U(T)\in C_b(\cX)$ and $\S$ is a $C_b$-semigroup, the first term is jointly continuous. By Lemma \ref{l.measurable}(c), 
also the convolution term is continuous. Thus, $U$ is continuous, hence a continuous martingale. By Theorem \ref{t.continuousmartingale},
$U$ solves Equation \eqref{eq:FVP:A*}.
\end{proof}

\begin{cor}\label{cor:quadraticvariation}
Let $\Psi\in\mathfrak{L}^1(0,T;C_b(\cX))\cap C_{b,\mathrm{loc}}([0,T)\times \cX)$ be adapted.
If $M(t):=\int_0^t \Psi(s)\,ds$ is an $\E$-martingale, then $M\equiv 0$.
\end{cor}
\begin{proof}
On the one hand, since $M$ is an $\E$-martingale, it follows that $\partial_t^\E M \equiv 0$. On the other hand,
\[
\partial_t^\E M(t) = \lim_{h \downarrow 0} \E_t \bigg[ \frac{1}{h} \int_t^{t+h} \Psi(s)\, ds \bigg] = \Psi(t) \quad \text{for all } t \in [0,T).
\]
Hence, $\Psi \equiv 0$, and therefore $M \equiv 0$.
\end{proof}

\begin{example}
Let $G(\x):=\int_0^T \chi(s) f(\x(s))\, ds$ for some $\chi\in C_b([0,T])$ and $f\in C_b(X)$. Then, 
\[
M(t):=\E_t G=\int_0^t \chi(s) f(\x(s))\,ds + \E_t \int_t^T \chi(s) f(\x(s))\,ds
\]
is an $\E$-martingale. Hence, by Theorem~\ref{t:Edifferentiable}, $V(t):=\E_t \int_t^T \chi(s) f(\x(s))\,ds$ is continuously $\E$-differentiable with $\partial_t^\E V(t):=-\chi(t)f(\x(t))$. We emphasize that the $\E$-derivative is independent of the expectation operator.
\end{example}

\begin{example}[It\^o's lemma]\label{ex:Ito simple}
Suppose that $\mathbb{W}$ is the expectation operator corresponding to the Brownian motion from Example \ref{ex.wieneroperator}. 
Let $V(t,\x):=u(t,\x(t))$ for some $u\in C^{1,2}_b([0,T]\times \R)$ and 
\[
\Psi(t,\x):=\partial_t u(t,\x(t))+\frac{1}{2}\partial^2_x u(t,\x(t)).
\]
Then, $\Psi\in C_b([0,T]\times\cX)$ %and $\partial_t\Phi\in  C_b([0,T]\times\cX)$
and thus satisfies the assumptions of Theorem~\ref{t:Edifferentiable}(ii). A direct application of It\^o's lemma shows that 
\[
M(t):=V(t)-\int_0^t\Psi(s)\,ds\quad\mbox{is a $\mathbb{W}$-martingale}.
\]
Hence, we can apply Theorem~\ref{t:Edifferentiable} and obtain that $V$ is continuously $\mathbb{W}$-differentiable with 
\[
\partial_t^\mathbb{W} V(t,\x)=\partial_t u(t,\x(t))+\frac{1}{2}\partial^2_x u(t,\x(t)).
\]
Moreover, as a result of Corollary~\ref{cor:quadraticvariation},
$V$ is a continuous $\mathbb{W}$-martingale if and only if $\Psi\equiv 0$, in which case $\partial_t u=-\frac{1}{2}\partial^2_x u$. 
%\RD{Hier muss es wohl $\partial_t u = - \frac12 \partial_x^2 u$ heißen.}
\end{example}

A more general version of the previous example, based on the functional It\^o formula, is discussed in Subsection~\ref{sec:ver Dupire} below.

\begin{example}[It\^o's isometry]
Let $M\colon [0,T] \to C_b(X)$ be a continuous $\mathbb{E}$-martingale with absolutely continuous quadratic variation
\[
\langle M \rangle_t = \int_0^t \Psi^2(s)\, ds,
\]
for an adapted  function $\Psi \in C_{b, \mathrm{loc}}([0,T)\times \cX)$ with $\Psi^2 \in \mathfrak{L}^1(0,T; C_b(\cX))$. That is, we assume that the process $M^2 - \langle M \rangle$ is an $\mathbb{E}$-martingale. It follows from Theorem~\ref{t:Edifferentiable} that
\[
\partial_t^{\mathbb{E}} M^2 \equiv \Psi^2.
\]
Let $\mathscr{S}$ denote the space of simple processes $H\colon [0,T] \to C_b(\cX)$ of the form
\[
H = \sum_{k=0}^{n-1} H_k \one_{(t_k, t_{k+1}]},
\]
for some partition $0 = t_0 < t_1 < \cdots < t_n = T$ and $H_k \in C_b(\cX; \cF_{t_k})$. As usual, for such $H \in \mathscr{S}$, the corresponding stochastic integral is defined by
\[
(H\cdot M)(t) := \sum_{k=0}^{n-1} H_k \big(M(t_{k+1} \wedge t) - M(t_k \wedge t)\big).
\]
By the definition of the $\mathbb{E}$-derivative, we obtain
\[
\partial_t^{\mathbb{E}} (H\cdot M)^2(t) = H^2(t)\, \partial_t^{\mathbb{E}} M^2(t) = H^2(t)\, \Psi^2(t) \quad \text{for all } t \in [0,T).
\]
As a consequence of Theorem~\ref{t:Edifferentiable}, we conclude that the process
\begin{equation}\label{mart:HM2}
(H\cdot M)^2(t) - \int_0^t H^2(s) \Psi^2(s)\, ds
\end{equation}
is an $\mathbb{E}$-martingale.
%, and that the process $U(t) := \Theta_{-t} (H\cdot M)^2(t)$ is a mild solution of  
%$\partial_t U(t) = -\A U(t) + \Theta_{-t} \int_0^t \vartheta^2(s) \Psi^2(s)\, ds$.
In particular, we obtain the It\^o isometry
\[
\mathbb{E}\big[(H\cdot M)^2(T)\big] = \mathbb{E} \bigg[ \int_0^T H^2(s) \Psi^2(s)\, ds \bigg].
\]
This identity allows the standard extension of the stochastic integral to the closure of $\mathscr{S}$ with respect to the norm $\|H\| := \big( \mathbb{E} \big[ \int_0^T H^2(s) \Psi^2(s)\, ds \big] \big)^{1/2}$.
\end{example}

\subsection{Dupire time derivative}
In this subsection, we discuss the relationship between the $\mathbb{E}$-derivative and the Dupire horizontal derivative as introduced in \cite{dupire}. 

\begin{defn}
Let $V\colon [0,T] \to C_b(\cX)$ be measurable and adapted. We say that $V$ is \emph{differentiable in the Dupire sense} if, for every $t \in [0,T)$, the limit
\begin{equation}\label{eq:DUP}
[\dt V(t)](\x) \coloneqq \lim_{h \downarrow 0} \frac{V(t+h, \tau_t(\x)) - V(t, \x)}{h}
\end{equation}
exists in the bp-sense. The function $\dt V$ is called the \emph{Dupire time derivative} or the \emph{horizontal derivative} of $V$.
If, in addition, $\dt V$ belongs to $C_{b,\mathrm{loc}}([0,T) \times \cX)$, we say that $V$ is continuously differentiable in the Dupire sense.
\end{defn}
For the horizontal derivative in \cite{dupire}, it is only required that the difference quotient in \eqref{eq:DUP} exists pointwise. For our purposes, we work with the slightly stronger notion of bp-convergence, which also requires that the difference quotient is uniformly bounded.

It turns out that the Dupire time derivative is the $\E$-derivative for a particular expectation operator,
namely the \emph{stopping operator} $\mathbb{\Lambda}\colon B_b(\cX)\to B_b(\cX)$, defined by
\[
(\mathbb{\Lambda} F)(\x) := F(\tau(\x)).
\]
We  also set $\mathbb{\Lambda}_t := \Theta_t \mathbb{\Lambda} \Theta_{-t}$ and observe that $\mathbb{\Lambda}_t F = F \circ \tau_t$.
We note that this is a special case of the situation in Example \ref{ex.detmart} with $\varphi=\tau$. In particular, we have
\[
\dt = \partial_t^\mathbb{\Lambda}.
\]
We collect some properties for later use.

\begin{lem}\label{l.dupireshift}
The stopping operator $\mathbb{\Lambda}$ defines a homogeneous expectation operator. Moreover, the family $(\mathbb{\Lambda} \Theta_t)_{t \ge 0}$ forms an evolutionary $C_b$-semigroup; its generator is denoted by $\mathbb{L}$.

The $C_b$-semigroup $(\mathbb{\Lambda} \Theta_t)_{t \ge 0}$ is Markovian and its induced $C_b$-semigroup on $C_b(X)$ is the identity with $C_b$-generator $L f = 0$ for all $f \in D(L) = C_b(X)$.
\end{lem}

\begin{proof}
Since $\tau(\x) = \tau^2(\x)$ and $\tau(\vt_t \tau(\x)) = \vt_t \tau(\x)$ for all $t \ge 0$, the map $\tau$ is an evolution map in the sense of \cite[Definition 6.1]{dkk}. By \cite[Proposition 6.2]{dkk}, $\mathbb{\Lambda}$ is a homogeneous expectation operator. Hence, $(\mathbb{\Lambda} \Theta_t)_{t \ge 0}$ is an evolutionary semigroup on $B_b(\cX; \mathcal{F}_0)$. It is even a $C_b$-semigroup, since the map $(t, \x) \mapsto [\mathbb{\Lambda} \Theta_t F](\x)$ is continuous for all $F \in C_b(\cX; \mathcal{F}_0)$. 

For the last part, note that for every $f \in C_b(X)$, we have
$[\mathbb{\Lambda} \Theta_t F_0(f)](\x) = [\mathbb{\Lambda} F_t(f)](\x) = f(\x(0))$,
which shows that the induced Markovian $C_b$-semigroup on $C_b(X)$ is the identity.
\end{proof}

We note that $\tau(\x)$ is constant for $t\geq 0$ and thus, in particular, differentiable. Using this, we obtain from Example \ref{ex.detmart}:

\begin{lem}
    \label{l.dupirediff}
    Let $u \in C_b([0,T]\times X)$ and define $V(t,\x) \coloneqq u(t, \x(t))$. Then $V$ is continuously Dupire differentiable if and only if
    $u\in C^{1,0}_b([0,T]\times X)$, i.e.\ for every $x\in X$, the map $t\mapsto u(t,x)$ is differentiable. In this case
    $\dt V(t,\x) = (\partial_t u)(t, \x(t))$. 
    In particular, $V$ is a $\mathbb{\Lambda}$-martingale if and only if $\partial_tu\equiv 0$, i.e., $t\mapsto u(t, x)$ is constant
    for every $x\in X$.
\end{lem}

\begin{proof}
    It suffices to observe that for $\varphi=\tau$, the right-hand side of Equation \eqref{eq.deterministicderivative}
    becomes $h^{-1}(u(t+h, \x(t)) - u(t, \x(t))$. As $\x(t)$ may take every value in $X$, it follows that $V$ is Dupire differentiable 
    if and only if $u$ is right-differentiable in $t$ at every point $(t,x)\in [0,T]\times X$. Noting that a right-differentiable function
    with continuous derivative is already differentiable implies the claim. The addendum follows from Theorem \ref{t:Edifferentiable}.
\end{proof}

As an application of Theorem~\ref{t:Edifferentiable}, we obtain the following result.
\begin{cor}\label{c:Edifferentiable2}
Let $V\colon [0,T] \to C_b(\cX)$ be measurable and adapted. Then, the following are equivalent:
 \begin{enumerate}[\upshape (i)]
        \item $V$ is continuously differentiable in the Dupire sense with $\dt V\in \mathfrak{L}^1(0,T;C_b(\cX))$.
        \item There is an adapted process $\Psi \in C_{b,\mathrm{loc}}([0,T) \times \cX) \cap \mathfrak{L}^1(0,T; C_b(\cX))$ such that $M(t) := V(t) - \int_0^t \Psi(s)\, ds$ is a $\mathbb{\Lambda}$-martingale.
    \end{enumerate}  
 In this case, it holds $\Psi = \dt V$ and $U(t) := \Theta_{-t} V(t)$ is the unique continuous mild solution of the FVP
\begin{equation} \label{eq:FVP:A}
\partial_t U(t) = -\mathbb{L} U(t) + \Theta_{-t} \dt V(t) \quad \text{with} \quad U(T) = \Theta_{-T} V(T).
\end{equation}
\end{cor}

For the remainder of this subsection, let $\S$ be an evolutionary $C_b$-semigroup with expectation operator $\E$ and $C_b$-generator $\A$. By combining the previous result with those in Section~\ref{sect.martingale}, we derive a differential equation directly for an
$\E$-martingales $M(t)$, rather than the shifted version $\Theta_{-t}M(t)$. To do so, we exchange the usual time derivative $\partial_t$
for the Dupire derivative $\dt$.

\begin{thm}
    \label{t.martDupire}
Let $T > 0$, and let $V \in C_b([0,T] \times \cX)$ and $\Psi \in \mathfrak{L}^1(0,T; C_b(\cX))$ be adapted such that the process
\[
M(t) := V(t) - \int_0^t \Psi(s)\, ds
\]
is an $\E$-martingale. Suppose that:
\begin{enumerate}
[\upshape (a)]
\item $V$ is continuously differentiable in the Dupire sense with $\dt V \in \mathfrak{L}^1(0,T; C_b(\cX))$,
\item the pointwise classical derivative of $U(t):=\Theta_{-t} V(t)$ 
exists on $[0,T)$ and belongs to $C_{b,\mathrm{loc}}([0,T)\times X)$.
\end{enumerate}
Then, for every $t \in [0,T)$, it holds  $U(t) \in D(\mathbb{L}) \cap D(\A)$ and
\begin{equation}\label{eq:Dupire2}
\dt V(t) = \Theta_t (\mathbb{L} - \A) \Theta_{-t} V(t) + \Psi(t).
\end{equation}
\end{thm}

\begin{proof}
By assumption (a) and Corollary~\ref{c:Edifferentiable2}, $U$ is the unique continuous mild solution of the FVP~\eqref{eq:FVP:A}. By assumption (b) and Theorem~\ref{t.mildstrong}, $U$ is even the strong solution of \eqref{eq:FVP:A}, whence $U(t)\in D(\mathbb{L})$ for all $t\in [0,T)$.
Moreover, it follows from Theorem~\ref{t.continuousmartingale} that $U$ is the unique mild solution of
\begin{equation}\label{eq.6.12.1}
\partial_t U(t) = -\A U(t) + \Theta_{-t}\Psi(t).
\end{equation}
Subtracting Equation \eqref{eq:FVP:A} from Equation \eqref{eq.6.12.1} and integrating yields for $t\in [0,T)$  and $h>0$ small enough
\[
\frac{1}{h}\int_t^{t+h} \Theta_{-s}\left( \dt V(s) - \Psi(s)\right)\, ds =  \frac{1}{h} \int_t^{t+h} \mathbb{L} U(s)\, ds - \A\frac{1}{h} \int_t^{t+h}  U(s)\, ds.
\]
Using that $U$ is a strong solution of \eqref{eq:FVP:A} and the regularity of $\partial_t U$ and $\dt V$, we have $\mathbb{L} U \in C_{b,\mathrm{loc}}([0,T) \times \cX)$. Hence, we can let $h \downarrow 0$ and obtain from the closedness of the $C_b$-generator $\A$ that
\[
U(t) \in D(\mathbb{A}) \quad \text{and} \quad \Theta_{-t} \dt V(t)  = \mathbb{L} U(t) - \A U(t) + \Theta_{-t}\Psi(t)
\]
for all $t \in [0,T)$. Applying the shift $\Theta_t$ to both sides leads to equation~\eqref{eq:Dupire2}.
\end{proof}

\begin{example}
    We illustrate Theorem \ref{t.martDupire} in the setting of Example \ref{ex.integralmeanmarkovian}. Thus, given
    $G=F_0^T(f)$ for $f\in C_b(X)$, we set $M(t) \coloneqq \E_t G$. Using the decomposition $G= F_0^t(f) + F_t^T(f)$, it follows
    that $M$ is of the form \eqref{eq.6.12.1} with $V(t) = \E_t F_t^T(f)$ and $\Psi(t, \x) = f(\x(t))$. If we denote
    the restriction of $\S$ to $B_b(\cX, \cF(\{0\}))$ by $S$, it follows that
    \[
    U(t) \coloneqq \Th_{-t} V(t) = \E F_0^{T-t}(f) = \int_0^{T-t}F_0(S(s)f)\, ds = F_0\Big(\int_0^{T-t}S(s)f\, ds \Big).
    \]
    For the $C_b$-generator $A$ of $S$, we have $A\int_0^{T-t}S(s)f = S(T-t)f - f$ by Lemma \ref{l.cbgen}(c). 
    Consequently, \eqref{eq:DMarkov} yields
    \[
    [\A U(t)](\x) = \Big[F_0\Big(A\int_0^{T-t}S(s)f\, ds \Big)\Big](\x) = [S(T-t)f](\x(0)) - f(\x(0)).
    \]
    On the other hand, $\mathbb{L} U(t) = F_0(L \int_0^{T-t}S(s)f) = 0$ by Lemma \ref{l.dupireshift}. It thus follows that
    \[
    \big[\Theta_t (\mathbb{L}-\A) \Theta_{-t}V(t)\big](\x) + \Psi(t, \x) = [S(T-t)f](\x(t)).
    \]
    Noting that $V(t,\x) = \int_0^{T-t}(S(s)f)(\x(t))\, ds$, we find
    \[
    \frac{V(t+h, \tau_t(\x))-V(t,\x)}{h} = \frac{1}{h}\int_{T-t-h}^{T-t}S(s)f(\x(t))\, ds \to [S(T-t)f](\x(t))
    \]
    as $h\to 0$. This verifies Theorem \ref{t.martDupire} directly in this example.
\end{example}

As an application of the previous result, we obtain the following characterization and regularity result in the particular case when $V(t,\x)=u(t,\x(t))$. We highlight that we only assume differentiability in time for $u$, and that spatial regularity is automatically obtained via the martingale condition.

\begin{cor}
    \label{c.martDupire}
Let $T > 0$, and let $u\in C^{1,0}_b([0,T]\times X)$
and $\Psi \in \mathfrak{L}^1(0,T; C_b(\cX))$ be adapted. 
If the process $[0,T]\times \cX\to\R$ given by
\[
u(t,\x(t)) - \int_0^t \Psi(s,\x)\, ds
\]
is an $\E$-martingale, then the following properties hold:
\begin{enumerate}
[\upshape (a)]
\item $V(t,\x):=u(t,\x(t))$ is continuously differentiable in the Dupire sense with $\dt V(t,\x)=\partial_t u(t,\x(t)) \in  C_{b,\mathrm{loc}}([0,T) \times \cX) \cap \mathfrak{L}^1(0,T; C_b(\cX))$.
\item The pointwise classical derivative of $U(t):=\Theta_{-t} V(t)$ 
exists on $[0,T)$, is given by $\partial_t U(t)=\Theta_{-t}\dt V(t)$ and belongs
to $ C_{b,\mathrm{loc}}([0,T) \times \cX) \cap \mathfrak{L}^1(0,T; C_b(\cX))$.
\item For every $t\in[0,T)$, the map $\x\mapsto u(t,\x(0))$ belongs to the domain $D(\A)$. In particular, if $\S$ is Markovian with induced $C_b$-generator $A$, we obtain that $u(t,\cdot)\in D(A)$ for all $t\in[0,T)$.
\end{enumerate}
In this case, it holds
\begin{equation}\label{eq:Dupire3}
\partial_t u(t,\x(t)) = -\Theta_t \A F_0(u(t)) + \Psi(t)
\end{equation}
for all $t \in [0,T)$, where we recall that $F_0(u(t))$ is the map $\x \mapsto u(t, \x(0))$. In particular, if $\S$ is Markovian with induced $C_b$-generator $A$, the equation \eqref{eq:Dupire3} becomes $\partial_t u(t,\x(t)) = -A u(t,\x(t)) + \Psi(t,\x)$ for all $t\in[0,T)$.
\end{cor}
\begin{proof}
First, it follows from Lemma \ref{l.dupirediff} that $V$ is continuously differentiable in the Dupire sense, with
\[
\dt V(t,\x)=\partial_t u(t,\x(t))\quad\text{for all }(t,\x)\in[0,T)\times \cX,
\]
and that $\partial_t u(t,\x(t))\in C_b([0,T]\times\cX)$. Consequently $\dt V\in  C_{b,\mathrm{loc}}([0,T) \times \cX) \cap \mathfrak{L}^1(0,T; C_b(\cX))$ and  we can apply Corollary~\ref{c:Edifferentiable2} to infer that $U$ is the continuous mild solution of the FVP \eqref{eq:FVP:A}. 
Actually, it follows from Theorem \ref{t.mildstrong}(b), that $U$ is a strong solution. Since
\[
U(t,\x)=u(t,\x(0))=F_0(u(t)),
\]
it follows from Lemma~\ref{l.dupireshift} that $\mathbb{L} U(t)=F_0(Lu(t))=0$. Hence, \eqref{eq:FVP:A} ensures that $U$ is pointwise classically differentiable with
\[
\partial_t U(t)=\Theta_{-t}\dt V(t)\quad\text{for all }t\in[0,T),
\]
and $\partial_t U\in  C_{b,\mathrm{loc}}([0,T) \times \cX) \cap \mathfrak{L}^1(0,T; C_b(\cX))$.

In a second step, we can apply Theorem~\ref{t.martDupire} and obtain that for every $t \in [0,T)$, it holds that $U(t) \in D(\A)$ and
\[
\partial_t u(t,\x(t)) =\dt V(t) =- \Theta_t \A \Theta_{-t} V(t) + \Psi(t) =- \Theta_t \A F_0(u(t)) + \Psi(t).
\]
Finally, if $\S$ is Markovian, we have $\Theta_{t}\A  F_0(u(t))=F_t(Au(t))$.
\end{proof}

\begin{example}\label{ex:Ito simple2}
Let $\mathbb{W}$ be the expectation operator corresponding to Brownian motion from Example~\ref{ex.wieneroperator}. Let $V(t,\x):=u(t,\x(t))$ for some $u \in C^{1,0}_b([0,T]\times \R^d)$. Suppose that $V$ is an $\E$-martingale. Then, it follows from Corollary~\ref{c.martDupire} that $u(t)\in D(\frac{1}{2}\Delta)$ for all $t\in[0,T)$, where $\frac{1}{2}\Delta$ is the $C_b$-generator of the Gaussian semigroup $G$ induced by the Markovian Gaussian semigroup $\mathbb{G}$. Moreover, $u$ solves the heat equation $\partial_t u(t)=-\frac{1}{2}\Delta u(t)$
for all $t \in [0,T)$.

In contrast to Example~\ref{ex:Ito simple}, where we relied strongly on It\^o's lemma, we do not a priori assume $C^2$-regularity in the second variable here.
\end{example}

\subsection{The vertical Dupire derivative}\label{sec:ver Dupire}
In this subsection, we discuss the connection of $\E$-martingales and the vertical Dupire derivative which is given by the It\^o formula for path-dependent functionals. We discuss this in the setting of  
It\^o diffusions, generalizing Example~\ref{ex:Ito simple}. 
As in Example~\ref{ex.wieneroperator}, let $B =(B_t)_{t\geq 0}$ be a $d$-dimensional Brownian motion defined on a probability space $(\Omega, \Sigma, P)$, and let $X= \R^d$ and $\cX= C(\R; \R^d)$. 
% For this, let $\Omega = C_0([0,\infty)) = \{ \omega\in C([0,\infty)): \omega(0)=0\}$, 
% $P$ be the Wiener measure on $\Omega$, 
%  $(B_t)_{t\ge 0}$ with $B_t(\omega) = \omega(t)$ be the  one-dimensional Brownian motion on the 
% probability space $(\Omega,\mathcal F,P)$, and let $(\mathcal F_t)_{t\ge 0}$ be the filtration generated by $\{B_s: s\le t\}$. 
% Assume $b\colon\R^d \to\R^d $  and $\sigma\colon \R^d\to\R^{d\times d}$ to be measurable functions satisfying
% \begin{align*}
% %    |b(x)| + |\sigma(x)| & \le C (1+|x|) ,\\
%     |b(x)-b(y)| + |\sigma(x)-\sigma(y)| & \le C |x-y|
% \end{align*}
% for all $x,y\in \R^d $ with some constant $C>0$. 
Assume $b\colon\R^d \to\R^d $  and $\sigma\colon \R^d\to\R^{d\times d}$ to be bounded and Lipschitz continuous functions, and assume that $\sigma(x)\sigma(x)^\top$ is
uniformly positive definite. Let $T>0$. Then, for every $x\in\R^d $ there 
exists a unique strong solution $(X_t^x)_{t\in [0,T]}$ of the SDE
\begin{equation}
    \label{eq-sde}
    \left\{
    \begin{aligned}
     dX_t^x &= b(X_t^x) dt + \sigma (X_t^x) dB_t,\\
     X_0^x & = x    
    \end{aligned}
    \right.
\end{equation}
which is adapted with respect to the filtration generated by $\{B_s: s\le t\}$  and which 
satisfies
\[ E \Big( \int_0^T |X_t^x|^2 \mathrm{d}t \Big)<\infty \]
(see
\cite[Theorem~5.2.1]{Oeksendal03}). Here, $E$ stands for the expectation with respect to $P$. The process $(X_t^x)_{t\ge 0}$ is Markovian and called the It\^o diffusion related to $b$ and $\sigma$.

Given $\x\in \cX$ and $\omega\in \Omega$, we define (in the same way as in Example~\ref{ex.wieneroperator})
    \[
    [\x\otimes_0 X^{\x(0)}(\omega)](t) \coloneqq \begin{cases}
        \x(t), & t\leq 0\\
        X_t^{\x(0)}(\omega) , & t>0. 
    \end{cases}
    \]
    We define the operator $\E \in \cL(B_b(\cX))$ by setting for $F\in B_b(\cX)$
    \begin{equation}
        \label{eq.sdeop}
        [\E F](\x) \coloneqq E\big[F(\x\otimes_0 X^{\x(0)})\big].
    \end{equation}
    Then $\E$ is a homogeneous expectation operator, see \cite[Example 6.11]{dkk}, and we denote its  associated evolutionary semigroup by $\S$. Considering $F=\one_A$ for a Borel set $A\in \mathscr{B}(\cX)$, we see that for every $\x\in \cX$, the probability measure $\P^\x$  is given by 
    \[ \P^\x(A) = P\big(X^{\x(0)}\in A^+(\x)\big)\]
    with $A^+(\x) := \{ \y^+: \y\in A,\, \y^-=\x^-\}$. Here, we have set $\y^+:= \y|_{[0,\infty)}$ and $\y^-:=\y|_{(-\infty,0]}$
    (see also Example~\ref{Ex:4.5}). Therefore, an adapted function $M\colon[0,T]\to C_b(\cX)$ is a $\P^\x$-martingale if and 
    only if $t\mapsto M(t,\x\otimes_0 X^{\x(0)})$ is a $P$-martingale.

To establish a connection between $\E$-martingales and the vertical Dupire derivative, we have to extend the path space $\cX=C(\R;X)$ to  the space $\mathscr{X}_{\mathrm{D}}\coloneqq D(\R, X)$   of all \emph{c\`adl\`ag paths} from $\R$ to $X$, endowed with Skorohod's $J_1$-metric. For the definition of the metric in $\mathscr{X}_{\mathrm{D}}$, we refer to \cite[Appendix B.2]{dkk}.

\begin{defn}
Let $V: [0,T] \to C_b(\mathscr{X}_{\mathrm{D}})$ be adapted. We say that $V$ is \emph{vertically differentiable in the Dupire sense}
if for every $t\in [0,T]$, every $v\in X=\R^d$,  and every $\x\in\mathscr{X}_{\mathrm{D}}$ the limit
\[
\partial_{\x}^v V(t) \coloneqq \lim_{h\downarrow 0} \frac{V(t, \tau^t(\x)+ h v\one_{[t,\infty)}) - V(t,\x)}{h}
\]
exists. In this case, the function $\partial_{\x}^v V$ is called the \emph{vertical Dupire  derivative} of $V$ in direction $v$,
and we define 
\[ \nabla_\x V := \big( \partial_\x^{e_j} V\big)_{j=1,\dots,d}\,,\]
where $e_j,\,j=1,\dots, d$, denotes the $j$-th unit vector in $\R^d$. The second vertical derivative 
\[ \nabla_\x^2 V = \big( \partial_\x^{e_j}  \partial_\x^{e_k}V\big)_{j,k=1,\dots,d}\]
 is defined by iteration. The function $V$ is said to be of class $\C^{1,2}(\cX_{\mathrm D})$ if the derivatives $\dt V$, $\nabla_\x V$
and $\nabla_{\x}^2V$ exist and are jointly continuous in $(t,\x)\in [0,T]\times \cX_{\mathrm D}$ (see \cite[Section~1.3]{Cont-Perkowski19} for details). If additionally the restrictions of these derivatives to $\cX$ belong to $C_{b,\mathrm{loc}}([0,T)\times\cX)$,
we write $V\in \C^{1,2}(\cX_{\mathrm D})\cap \C^{1,2}_{b,\mathrm{loc}}(\cX)$.
In the following, $\operatorname{tr}(\cdot)$ stands for the trace of a matrix.
\end{defn}
% \color{blue}
% Is $\C^{1,2}(\cX_{\mathrm D})\cap \C^{1,2}(\cX)$ still a good name? Maybe better $\C^{1,2}(\cX_{\mathrm D})\cap \C^{1,2}_{b,\mathrm{loc}}(\cX)$
% \color{black}

\begin{thm}\label{Thm:Ito-Dupire}
    Let $V\colon [0,T]\to C_b(\cX)$ be adapted, and assume that for each $t\in [0,T]$, the function $V(t)$ admits an extension to $V(t)\in C_b(\cX_{\mathrm D})$ such that $V$ is of class $\C^{1,2}(\cX_{\mathrm D})\cap \C^{1,2}_{b,\mathrm{loc}}(\cX)$. Let $\E$ be the expectation operator related to the It\^o diffusion \eqref{eq-sde} as discussed above. 
    \begin{enumerate}
[\upshape (a)]
\item Define 
   \[ \Psi(t,\x) := \dt V(t,\x)+\langle b(\x(t)), \nabla_\x V(t,\x)\rangle 
   + \tfrac 12\operatorname{tr} \Big( \sigma(\x(t))\sigma(\x(t))^\top \nabla_\x^2 V(t,\x)\Big).\]
   Then $\Psi\in C_{b,\mathrm{loc}}([0,T)\times \cX) \cap \mathfrak{L}^1(0,T;C_b(\cX))$ is adapted 
   and  $M(t):=V(t)-\int_0^t \Psi(s)\,ds$ is an $\E$-martingale.
\item $V$ is continuously $\E$-differentiable with $\partial_t^\E V = \Psi$, and $U(t):= \Theta_{-t}V(t)$ is the unique continuous 
mild solution of the FVP \eqref{eq:FVP:A*}. 
\item $V$ is an $\E$-martingale if and only if for all $\x\in \cX$, $V$ is a solution of the Dupire equation
   \[ \dt V(t,\x) + \langle b(\x(t)), \nabla_\x V(t,\x)\rangle 
   + \tfrac 12\operatorname{tr} \Big( \sigma(\x(t))\sigma(\x(t))^\top \nabla_\x^2 V(t,\x)\Big)=0.\]
\end{enumerate}
\end{thm}

We remark that this result also yields a description of the operator $\A-\mathbb{L}$. Assume that in the 
situation of Theorem~\ref{Thm:Ito-Dupire}, the pointwise classical derivative of  $U(t):=\Theta_{-t} V(t)$ 
exists on $[0,T)$ and belongs to $C_{b,\mathrm{loc}}([0,T)\times X)$. Then we can apply Theorem~\ref{t.martDupire}
and get
\begin{align*}
& \big(\Theta_t (\A-\mathbb{L} ) \Theta_{-t} V\big) (t,\x) =  \Psi(t,\x)-\dt V(t,\x) \\
& \qquad =
\langle b(\x(t)), \nabla_\x V(t,\x)\rangle 
   + \tfrac 12\operatorname{tr} \Big( \sigma(\x(t))\sigma(\x(t))^\top \nabla_\x^2 V(t,\x)\Big).
\end{align*}
For $F:= U(t)$ with fixed $t\in [0,T]$, this implies
\[ (\A-\mathbb L)F(\x) = \langle b(\x(0)), \nabla_\x^0 F(\x)\rangle 
   + \tfrac 12\operatorname{tr} \Big( \sigma(\x(0))\sigma(\x(0))^\top (\nabla_\x^0)^2 F(\x)\Big).\]
Here, $\nabla_\x^0$ stands for the vertical Dupire derivative at $t=0$, i.e.,
\[ \nabla_{\x}^0 F (\x)\coloneqq \Big(\lim_{h\downarrow 0} \frac{F(\tau(\x)+ h e_j\one_{[0,\infty)})- F(\x)} {h}\Big)_{j=1,\dots,d}\,,\]
and $ (\nabla_\x^0)^2 $ is defined by iteration.

\begin{proof}[Proof of Theorem~\ref{Thm:Ito-Dupire}]
    (a). By the conditions on $V$ and on the coefficients $b$ and $\sigma$, we immediately get that $\Psi\in C_{b,\mathrm{loc}}([0,T)\times \cX)\cap \mathfrak{L}^1(0,T;C_b(\cX))$. 
    It is easily seen that $\nabla_\x V$ and $\nabla_\x^2V$ are 
    adapted.    To show that $M$ is an $\E$-martingale, we apply Lemma~\ref{l.pmartingale}. For this, we note that, as $\sigma\sigma^\top$ is uniformly positive definite, we can apply the support theorem of Stroock--Varadhan \cite[Theorem~3.1]{Stroock-Varadhan72} to see that $P\circ (X^{x})^{-1}$ has full support $\{\x^+\in \cX^+: 
    \x(0)=x\}$ for all $x\in X$. Therefore, we can argue as in Example~\ref{Ex:4.5} to see that 
    $\E|F|=0$ implies $F=0$ for $F\in C_b(\cX)$. 

    It remains to show that $M$ is a $\P^\x$-martingale for all $\x\in\cX$. As explained above, this is equivalent to the condition that $M(\cdot, \x\otimes_0 X^{\x(0)})$ is a $P$-martingale for all $\x\in\cX$. To prove this, we apply the pathwise 
    It\^o formula for path-dependent functionals (\cite[Theorem~3]{Cont-Fournie10}, see also \cite[Theorem~1.10]{Cont-Perkowski19}) and obtain for all paths $\y = \x\otimes_0 X^{\x(0)}$ 
    \begin{equation}\label{eq:6.8}
    \begin{aligned}
        V(t,\y) - V(0,\y) & = \int_0^t  \dt V(s,\y) \mathrm{d}s+ 
        \int_0^t \tfrac12  \operatorname{tr} \big( \nabla_\x^2 V(s,\y) \mathrm{d}[\y]_s\big) \\
        & \quad + \int_0^t \langle\nabla_\x V(s,\y) \mathrm{d}\y(s)\rangle .
    \end{aligned}
    \end{equation}
    For general paths $\y$ with finite quadratic variation, the last integral is defined as the limit over Riemannian sums. However, as we can see from \cite[Proposition~7]{Cont-Fournie10}, for 
    continuous semi-martingales, this coincides with the It\^o integral. Using \eqref{eq-sde}, the last integral
    in \eqref{eq:6.8} equals
    \[ \int_0^t \langle b(\y(s)), \nabla_\x V(s,\y)\rangle  \mathrm{d}s
    + \int_0^t \langle \nabla_\x V(x,\y), \sigma(\y(s)) \mathrm{d}B_s\rangle.\]
Moreover, it is known that for the It\^o diffusion, the quadratic variation
is given by 
\[  [\y]_t =  \int_0^t \sigma(\y(s))\sigma(\y(s))^\top \mathrm{d}s.\]
With this and the definition of $\Psi$, we see that 
\[ M(t,\y)-M(0,\y) = \int_0^t \langle \nabla_\x V(x,\y), \sigma(\y(s)) \mathrm{d}B_s\rangle\quad  \text{ for } \y=\x\otimes X^{\x(0)},\]
and therefore $M$ is a $\P^\x$ martingale for all $\x\in\cX$. Now we can apply Lemma~\ref{l.pmartingale} to get a).

(b). This follows from Theorem~\ref{t:Edifferentiable}.

(c). Due to (a), $V$ is an $\E$-martingale if and only if $t\mapsto \int_0^t\Psi(s)\mathrm{d}s$ is an $\E$-martingale. 
By Corollary~\ref{cor:quadraticvariation}, this is equivalent to $\int_0^\cdot \Psi(s)\mathrm{d}s \equiv 0$ and therefore to $\Psi\equiv 0$.
\end{proof}

\begin{rem}
   a)  The above theorem shows, in particular, that every solution of the Dupire equation is directly connected with a mild solution of \eqref{eq:FVP:A*} and therefore with the semigroup $\S$. Here, the rather strong assumptions on the smoothness of $V$ are only needed for the formulation of the Dupire equation and the application of the It\^o formula for path-dependent functionals, whereas the connection between $\E$-martingales and mild solutions hold in a very general setting. 

   b) In the case where the It\^o diffusion is given by the one-dimensional Brownian motion, one obtains the path-dependent heat equation 
   \[ \dt V(t,\x) + \partial_{\x}^2 V(t,\x) = 0\]
   which was considered, e.g.,  in \cite{Peng-Song-Wang23}, where we write $\partial_\x$ instead of $\nabla_\x$ 
   if $d=1$. 
\end{rem}

\begin{example}\label{Ex:Support-condition}
The following example illustrates the role of the support condition on $\E$. Consider \eqref{eq-sde} for $d=1$
with $\sigma\equiv 0$ and with $b\colon\R\to\R$ being a bounded and strictly decreasing Lipschitz function with
$b(0)=0$
(e.g., $b(x)=-\arctan x$). Then $X^x$ is the unique solution of the ordinary differential equation
\[ \frac{d}{dt}\, X^x_t = b(X^x_t),\quad X_0^x=x.\]
Due to the sign condition on $b$, we see that for initial value $x\in\R$, the function
$t\mapsto |X_t^x|$ is decreasing on $[0,\infty)$. We define, e.g., 
\[ V(t,\x) :=  \min\{ (|\x(t)|-|\x(0)|)^+, 1\}\]
for $t\in [0,T]$ and $\x\in\cX$. Then $V\colon[0,T]\to C_b(\cX)$ is adapted, and 
\[ V(t,\x\otimes_0 X^{\x(0)}) = 0\]
for all $t\in [0,T]$ and $x\in\cX$. Therefore, $V$ is a $\P^\x$-martingale for all $\x\in\cX$. However, for 
$\x(t)=t$, we get \[\E V(t,\x) = V(t, \x\otimes_0 X^{\x(0)}) = V(t,\x\otimes_0 0)=0 < \min\{t,1\} = V(t,\x)\]
for $t\in (0,T]$. Therefore, $V$ is no $\E$-martingale.

Similarly, let $V$ be a smooth function  as in Theorem~\ref{Thm:Ito-Dupire}, and assume that 
 $V$ is a $\P^\x$-martingale for all $\x\in \cX$. Then the Dupire equation in Theorem~\ref{Thm:Ito-Dupire} (c)
is satisfied for all paths of the form $\x\otimes_0 X^{\x(0)}$, but in general not for all $\x\in\cX$.
\end{example}

\end{document}